\documentclass[11pt, a4paper]{amsart}
\usepackage{standard}

\usepackage{graphicx}
\usepackage{color}
\usepackage{hyperref}
\counterwithin{figure}{section}
\definecolor{myblue}{rgb}{0,0,0.6}
\hypersetup{colorlinks=true,linkcolor=myblue,citecolor=myblue,filecolor=myblue,urlcolor=myblue}

\numberwithin{equation}{section}
\numberwithin{table}{section}
\numberwithin{figure}{section}

\newcommand{\cH}{\mathcal{H}}
\newcommand{\beqs}{\begin{equation*}}
\newcommand{\eeqs}{\end{equation*}}
\newcommand{\beq}{\begin{equation}}
\newcommand{\eeq}{\end{equation}}
\newcommand{\tand}{\text{ and }}
\newcommand{\tfa}{\text{ for all }}

\newcommand{\newy}{y}
\newcommand{\obst}{\mathcal{O}}
\newcommand{\Rea}{\mathbb{R}}

\newcommand{\cM}{\mathcal{M}}

\newcommand{\nvs}{|v|^2}
\newcommand{\ngvs}{|\nabla v|^2}

\newcommand{\gv}{\nabla v}
\newcommand{\vb}{\overline{ v}}
\newcommand{\bx}{x}

\newcommand{\pdiff}[2]{\frac{\partial #1}{\partial #2}}
\newcommand{\diff}[2]{\frac{d #1}{d #2}}

\newcommand{\smallp}{\rho}%\varpi}
\newcommand{\smallpt}{\widetilde{\smallp}}

\newcommand{\comp}{\operatorname{comp}}
\newcommand{\cF}{\mathcal{F}}
\newcommand{\hsc}{h}
\newcommand{\rd}{d}
\DeclareMathOperator{\RC}{\mathsf{RC}}
\newcommand{\resolve}{{\mathcal{R}}}

\usepackage{color,soul}
\definecolor{escol}{rgb}{0,0,0.8}
\newcommand{\N}[1]{\norm{#1}}

\usepackage[a4paper]{geometry}
\usepackage{xcolor}

\title[Helmholtz quasi-resonances are unstable]{
Helmholtz quasi-resonances are unstable under most single-signed perturbations of the wave speed
}
\date{\today}
\author{Euan A.~Spence}
\address{Department of Mathematical Sciences, University of Bath, Bath, BA2 7AY, UK}
\email{e.a.spence@bath.ac.uk}
\author{Jared Wunsch}
\address{Department of Mathematics, Northwestern University, Evanston, IL, USA}
\email{jwunsch@math.northwestern.edu}
\author{Yuzhou Zou}
\address{Department of Mathematics, Northwestern University, Evanston, IL, USA}
\email{yuzhou.zou@northwestern.edu}

\begin{document}

\begin{abstract}
We consider Helmholtz problems with a perturbed wave speed, where the single-signed perturbation is linear in a parameter $z$. 
Both the wave speed and the perturbation are allowed to be discontinuous (modelling a penetrable obstacle). 
We show that there exists a polynomial function of frequency such that, for any frequency, for most values of $z$, the norm of the solution operator is bounded by that function.

This solution-operator bound is most interesting for Helmholtz problems with strong trapping; recall that here there exists a sequence of real frequencies, tending to infinity, through which 
the solution operator grows superalgebraically, with these frequencies often called \emph{quasi-resonances}.
The result of this paper then shows that, at every fixed frequency in the quasi-resonance, the norm of the solution operator becomes much smaller for 
%the superalgebraic growth of the solution operator does not occur 
most single-signed perturbations of the wave speed, i.e., quasi-resonances are unstable under 
most such perturbations.

%The result of this paper then shows that, at every quasi-resonance, the superalgebraic growth of the solution operator does not occur for most single-signed perturbations of the wave speed, i.e., quasi-resonances are unstable under 
%most such perturbations.
\end{abstract}

\maketitle

\section{Introduction}

\subsection{The main results}

Let $\Delta$ be the Laplace operator on $\RR^d$, $d\geq 1$.
Let $n\in L^\infty(\mathbb{R}^d;\mathbb{R})$ be strictly positive 
and equal to $1$ outside a sufficiently-large ball. 
Let 
$\psi \in L^\infty_{\rm comp}(\mathbb{R}^d;\mathbb{R})$. 
Let $R_0>0$ be such that $\supp(1-n)\cup\supp(\psi)\subset B(0,R_0)$.

Given $k>0$, $z\in\mathbb{R}$, and $f\in L^2_{\rm comp}(\mathbb{R}^d)$, let $u\in H^1_{\rm loc}(\mathbb{R}^d)$ be the solution to
\beq\label{eq:Helmholtz}
(-k^{-2} \Delta -n - z\psi)u=f
\eeq
%where $u$ is outgoing if it 
satisfying the Sommerfeld radiation condition 
\beq\label{eq:src}
\left(k^{-1}\frac{\partial}{\partial r} - i\right)u = o(r^{-(d-1)/2}) \quad\text{ as } r:=|x|\to \infty, \text{ uniformly in } x/r
\eeq
(such a solution is said to be \emph{outgoing}). 
With the Helmholtz equation \eqref{eq:Helmholtz} understood as coming from the wave equation via Fourier transform in time (with Fourier variable $k$), $n+z\psi$ is then the inverse of the square of the wave speed.

The existence and uniqueness of the solution to \eqref{eq:Helmholtz}-\eqref{eq:src} is standard (see the recap in Lemma \ref{lem:allgood} below).
We then write $u = (-k^{-2} \Delta -n - z\psi-i0)^{-1}f$, where the
$i0$ indicates that the radiation condition can be obtained by the
limiting absorption principle (see Lemma~\ref{lem:allgood} below).

This paper is concerned with the behaviour of  the solution operator
$\chi (-k^{-2} \Delta -n - z\psi-i0)^{-1}\chi: L^2 \to L^2$, where
$\chi \in C^\infty_{\rm comp}(\mathbb{R}^d)$, as a function of both
the frequency $k>0$ and the perturbation parameter $z$.

We recall 
  that the high-frequency behaviour of the solution operator $\chi(-k^{-2}\Delta-n)^{-1}\chi$ with smooth $n$ is 
  closely linked to the dynamics of
the Hamiltonian system with Hamiltonian
\beq\label{eq:p}
p(x,\xi)=\abs{\xi}^2-n(x).
\eeq
Letting $\Sigma$ denote the characteristic set, a.k.a., the energy surface,
$$
\Sigma := \big\{(x,\xi): p(x,\xi)=0\big\},
$$
we consider the dynamics inside $\Sigma$ given by the flow along the
\emph{Hamilton vector field}
$$
H_p:= 2\xi \cdot \pa_x +\nabla n \cdot \pa_\xi.
$$
(This is Newton's second law, with force given by the gradient of the
potential $-n$.)  The integral curves of $H_p$, i.e., the solutions $(x(t),\xi(t))$ of 
\beqs
\diff{x_i}{t}
= 2 \xi_i \quad\text{ and } \quad 
\diff{\xi_i}{t}
 = \pdiff{n}{x_i},
\eeqs
 in $\Sigma$ are known
as \emph{null bicharacteristics}.  A null bicharacteristic is said to
be trapped forwards/backwards if
$$
\lim_{t \to\pm \infty} \abs{x(t)} \neq \infty. 
$$
We say that a set $S \subset \RR^d$ \emph{geometrically controls} the backward
trapped null bicharacteristics if for each $(x(0), \xi(0))$ on a backward trapped null
bicharacteristic $(x(t), \xi(t))$, there exists $T<0$ with $x(T) \in S$.

\begin{theorem}[Main result for smooth $n$]\label{thm:main}
Suppose that, in addition to the assumptions on $n$ and $\psi$ above, $n,\psi\in C^{\infty}$
and $\psi\geq c>0$ on a set that geometrically controls all
backward-trapped null bicharacteristics for $-k^{-2} \Delta -n$. 

(a) 
Given $\epsilon, k_0, \smallp>0$ and $\chi \in C^\infty_{\rm comp}(\mathbb{R}^d)$, there exists $C_1>0$ such that, for all $k\geq k_0$, 
$z\mapsto \chi (-k^{-2} \Delta -n-z\psi-i0)^{-1}\chi$ extends meromorphically from $z\in (-\smallp,\smallp)$ to $z\in\mathbb{C}$, with the number of poles 
 in $\{ z \in \mathbb{C}, |z|< \smallp\}$ bounded by $C_1 k^{d+1+\epsilon}$. 

(b) 
There exists $\smallp>0$ such that the following is true. Given $\epsilon, k_0,\delta>0$, $N\geq 0$, and $\chi \in C^\infty_{\rm comp}(\mathbb{R}^d)$, there exists $C_2>0$ such that for all $k\geq k_0$
there exists a set $S_k\subset (-\smallp,\smallp)$ with Lebesgue measure $|S_k|
\leq \delta k^{-N}$ such that 
\begin{equation}
\label{eq:resolvent}
\| \chi (-k^{-2} \Delta -n- z\psi-i0)^{-1}\chi\|_{L^2 \to L^2} \leq \frac{C_2}{ \delta} k^{5(d+1)/2 +N+ \epsilon}
\text{ for all } z \in (-\smallp,\smallp)\setminus S_k.
\end{equation}
\end{theorem}

When $N=0$ in Part (b) of Theorem \ref{thm:main}, the set of $z$
excluded in \eqref{eq:resolvent} has arbitrarily small measure, independent of $k$. Choosing $N>0$ allows one to decrease the measure of this excluded set as $k$ increases (at the price of a larger exponent in the bound).

\begin{theorem}[Main result for discontinuous $n$]\label{thm:mainT}
Given a bounded open set $\obst$ with Lipschitz boundary and connected complement and $n_i>0$, let 
\beq\label{eq:n}
n := 
\begin{cases}
n_i & \text{ in } \obst,\\
1 & \text{ in } \mathbb{R}^d\setminus \overline{\obst}
\end{cases}
\eeq
Suppose that $\supp \psi \supset \obst$ and there exists $c>0$ such that $\psi \geq c$ on $\obst$ (observe that this includes the case $\psi = 1_\obst$). 

(a) Given $\smallp, k_0>0$ and $\chi \in C^\infty_{\rm comp}(\mathbb{R}^d)$, there exists $C_1>0$ such that, for all $k\geq k_0$, 
$z\mapsto \chi (-k^{-2} \Delta -n-z\psi-i0)^{-1}\chi$ extends meromorphically from $z\in (-\smallp,\smallp)$ to $z\in\mathbb{C}$, with the number of poles 
 in $\{ z \in \mathbb{C}, |z|< \smallp\}$ bounded by $C_1 k^{d+2}$. 

(b) 
Given $ \smallp, \epsilon, k_0,\delta>0$, $N\geq 0$, and $\chi \in C^\infty_{\rm comp}(\mathbb{R}^d)$, there exist $C_2>0$ such that for all $k\geq k_0$
there exists a set $S_k\subset (-\smallp,\smallp)$ with $|S_k|\leq \delta k^{-N}$ such that 
\begin{equation}
\label{eq:resolventT}
\| \chi (-k^{-2} \Delta -n- z\psi-i0)^{-1}\chi\|_{L^2 \to L^2} \leq \frac{C_2}{ \delta} k^{2 + 5(d+3)/2 +N+ \epsilon}
\text{ for all } z \in (-\smallp, \smallp)\setminus S_k.
\end{equation}
\end{theorem}

We highlight that the dynamical assumption in Theorem \ref{thm:mainT} is related to that  in
  Theorem~\ref{thm:main}.
Indeed, in the case of discontinuous, piecewise-constant $n$ (as in \eqref{eq:n}), the dynamics of null bicharacteristics is
  that of straight-line motion (arising from a constant potential) away from the
  obstacle, but \emph{internal refraction} at the boundary can produce trapping
  of rays in the manner of whispering gallery solutions (see, e.g., \cite{PoVo:99} and \S\ref{sec:QR} below).  
The assumption of Theorem \ref{thm:mainT} that $\psi$ is strictly
  positive on the obstacle therefore implies that it is strictly positive on all the backwards trapped null bicharacteristics, 
 and hence geometrically controls them. 

\subsection{Application of the main results to problems with quasi-resonances}\label{sec:QR}

When $n(x)-1$ decays sufficiently quickly as $|x|\to \infty$,  
%is a sufficiently-quickly decreasing function of $|x|$, 
then $\chi (-k^{-2} \Delta -n-i0)^{-1}\chi$ has poles (i.e., resonances) as a function of (the complex extension of) $k$ that are close to the real axis. 

Indeed, \cite[Theorem 1]{Ra:71} showed that if $n\in C^\infty$ is radial and 
$|x| \sqrt{n(|x|)}$ is not monotonically increasing (i.e., at some point $n(x)$ decreases faster than $|x|^{-2}$), then there exist 
a sequence of poles of $\chi (-k^{-2} \Delta -n-i0)^{-1}\chi$
exponentially close to the real axis; i.e., there exist poles $\{z_j\}_{j=1}^\infty$ such that $z_j\in\mathbb{C}$ with $\Re z_j \geq 1$, $|z_j|\to\infty$ as $j\to\infty$, 
and $|\Im z_j| \leq C_1 \exp( -C_2 |z_j|)$ for all $j$.
 In the penetrable-obstacle case, if $\obst$ is smooth and uniformly convex and $n_i>1$ then \cite[Theorem 1.1]{PoVo:99} showed that 
there exist a sequence of poles 
of $\chi (-k^{-2} \Delta -n-i0)^{-1}\chi$
superalgebraically close to the real axis; these are related to the classic ``whispering-gallery modes" (see, e.g., \cite{BaBu:91}).

By \cite[Theorem 1]{St:00}, the existence of poles  superalgebraically close to real axis implies the existence of quasimodes with superalgebraically small error; i.e., in both the cases mentioned above,
 there exist $\{k_j\}_{j=1}^\infty$, with $k_j\to \infty$ as $j\to\infty$, such that, given $N>0$ there exists $C_N>0$ such that, for all $j$, 
\begin{equation}\label{eq:superalgebraic}
\| \chi (-k_j^{-2} \Delta -n-i0)^{-1}\chi \|_{L^2\to L^2} \geq C_N (k_j)^N;
\end{equation}
these $k_j$ are often called \emph{quasi-resonances}.
(Note that, in the penetrable-obstacle case when $\obst$ is a ball, exponential growth of the solution operator through the quasi-resonances is shown in \cite{Ca:12, CaLePa:12, AlCa:18}.)

The bounds \eqref{eq:resolvent} and \eqref{eq:resolventT} applied to the two situations above with $k=k_j$ show that, 
for fixed $j$ sufficiently large (depending on the size of $C_N$), %, 
%at every fixed frequency in the quasi-resonance, 
for most
$z$ (more precisely, for all $z$ apart from a set of arbitrarily small measure), the bound \eqref{eq:superalgebraic} does not hold. 
%the norm of the solution operator becomes much smaller than \eqref{eq:superalgebraic}
%the superalgebraic growth 
%\eqref{eq:superalgebraic} is disrupted by the perturbation 
%for most
%$z$ (more precisely, for all $z$ apart from a set of arbitrarily small measure). 
%; i.e., at a quasi-resonance, the superalgebraic growth does not occur for most single-signed perturbations of the wave speed.
%This solution-operator bound is most interesting for Helmholtz problems with strong trapping; recall that here there exists a sequence of real frequencies, tending to infinity, through which 
%the solution operator grows superalgebraically, with these frequencies often called \emph{quasi-resonances}.
%The result of this paper then shows that, at every fixed frequency in the quasi-resonance, the norm of the solution operator becomes much smaller for 
%%the superalgebraic growth of the solution operator does not occur 
%most single-signed perturbations of the wave speed, i.e., quasi-resonances are unstable under 
%most such perturbations.
%Highlight that quasi-resonances are still there, but 

We highlight that our result does not show that quasi-resonances \emph{disappear} under perturbation. For instance, for a penetrable obstacle, for each (fixed) perturbed wave speed, there still exist frequencies through which \eqref{eq:superalgebraic} holds (by  \cite{St:00, Ca:12, CaLePa:12, AlCa:18}).

The instability of %the growth through 
quasi-resonances is illustrated qualitatively for low frequencies in Figures \ref{fig:QR1} and \ref{fig:QR2} 
in the setting of Theorem \ref{thm:mainT}. Figures \ref{fig:QR1} and \ref{fig:QR2} both plot the absolute value of the wave scattered by an incident plane wave for $n=n_i 1_\obst$, $\psi=1_\obst$, and $\obst =B(0,1)$. 
In this case, the solution can be written down explicitly in terms of Fourier series and Bessel/Hankel functions, and the quasi-resonances expressed as zeros of a combination of Bessel/Hankel functions. At least for small $k$, both the quasi-resonances and the solution can thereby be computed accurately; this was done in \cite[\S6.2]{MoSp:19}, and Figures \ref{fig:QR1} and \ref{fig:QR2} are plotted using the same MATLAB code.

%Both figures plot the scattered wave produced by the incident plane wave $u^I(x):= \exp(i k x \cdot a)$ where $a= (\cos (\pi/6), \sin(\pi/6))$
%
%first frequency
%$k=0.992772133752486$ %=8,1
%
%second frequency 
%$k=2.19476917403094$ %=15,2

\begin{figure}[h]
%\centering
\includegraphics[width=.49\textwidth%width=70mm
, clip, trim={80 30 45 30}]{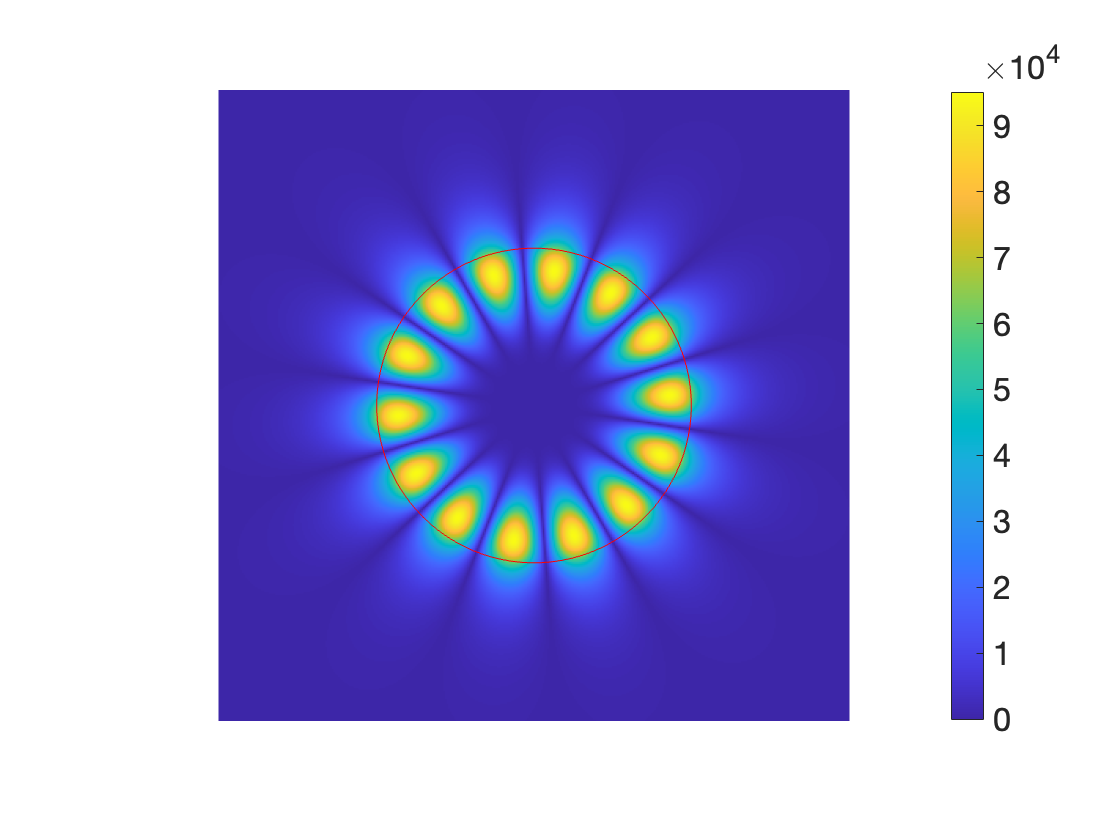}
%{{TransmissionProblem_PW,_k=001.77945199481921,_ni=100,_no=1_PAP_gimp}.eps}
\includegraphics[width=.49\textwidth%width=70mm
, clip, trim={80 30 45 30}]{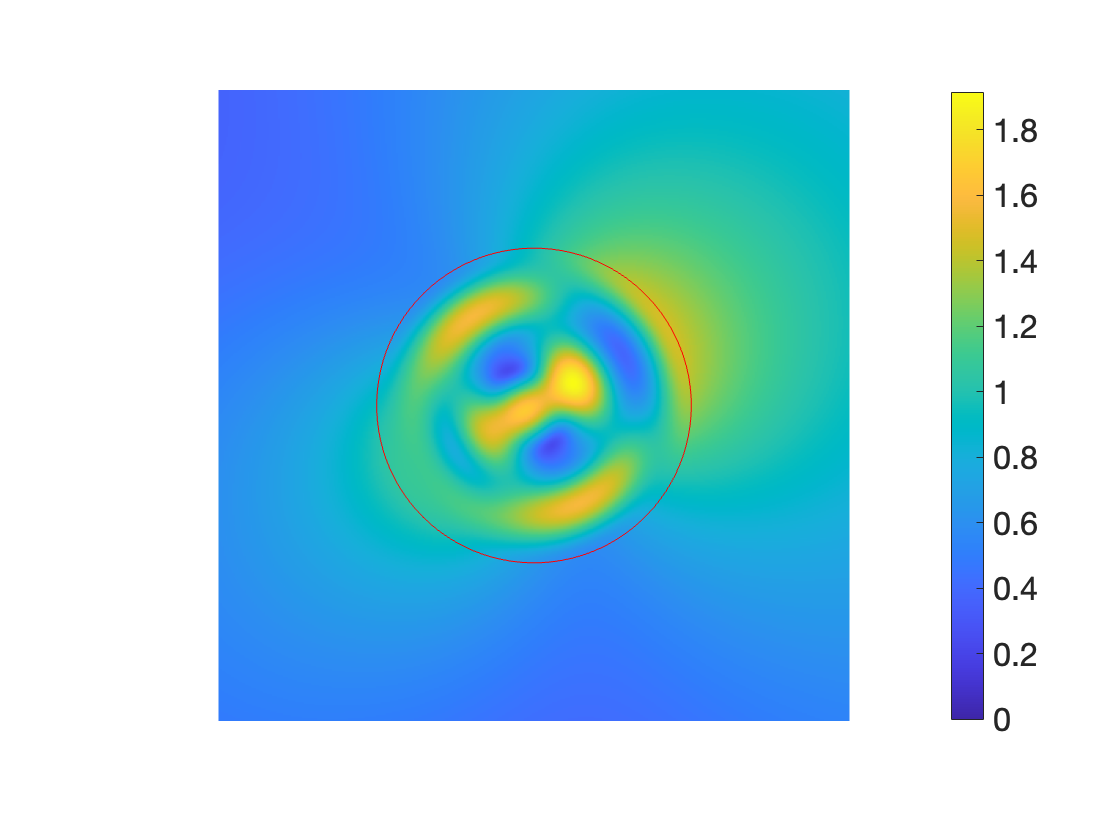}
%$k=0.992772133752486$\\
$n_i=100$, $z=0$ \hspace{4cm} $n_i=100$, $z=0.01$ 

\caption{The absolute value of the field scattered by the plane wave $\exp(i k(x\cos(\pi/6) + y\sin(\pi/6)))$ 
hitting the penetrable obstacle $\obst = B(0,1)$ with $n_i=100$ 
(the red line denotes the boundary $\obst$.) The frequency is $k=0.992772133752486$, which is (an approximation to) a quasi-resonance for the problem when $n_i=100$. The left plot corresponds to $n_i=100$ and $z=0$ (i.e., the setting of the quasi-resonance) and the right plot corresponds to $n_i=100$ and $z=0.01$ (i.e., a small perturbation of the wave speed). We highlight the different scales on the colour bars. 
}\label{fig:QR1}
\end{figure}

\begin{figure}[h]
\includegraphics[width=.49\textwidth%width=70mm
, clip, trim={80 30 45 30}]
{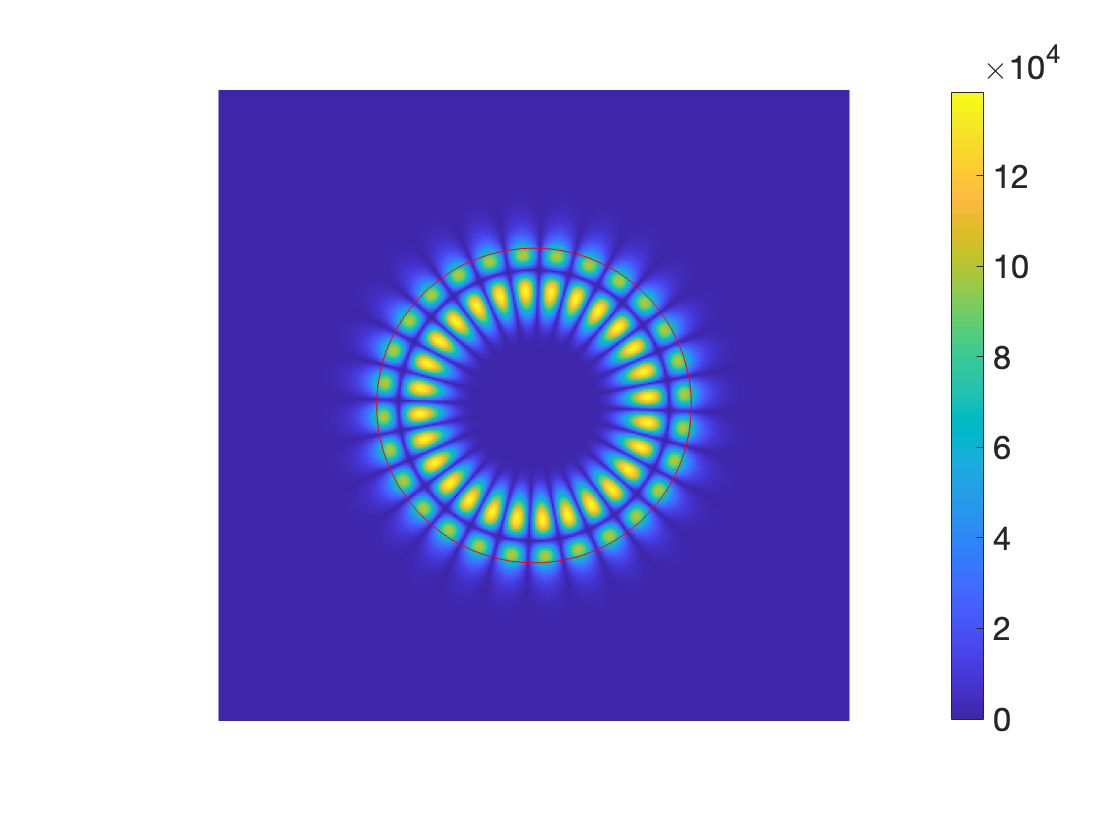}
%{figures/second_frequency_n_100/figure1.png}
\includegraphics[width=.49\textwidth%width=70mm
, clip, trim={80 30 45 30}]
{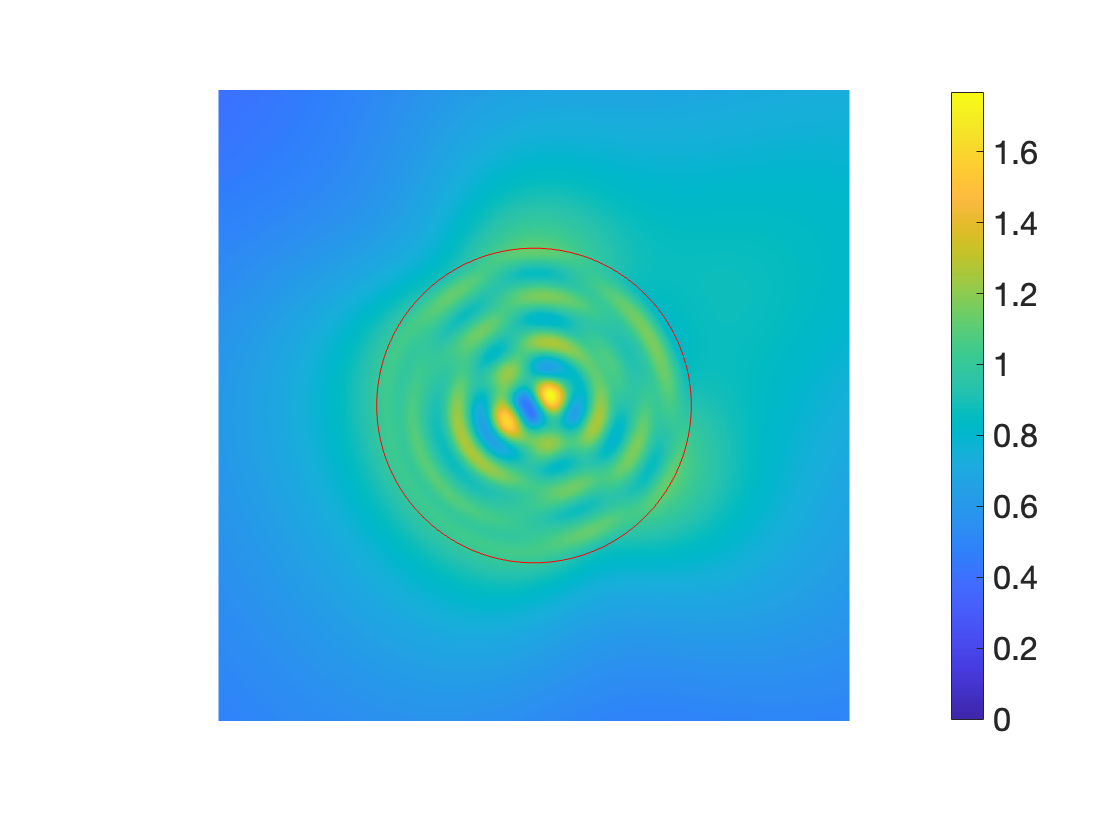}
%{figures/second_frequency_n_100point01/figure1.png}
%$k=2.19476917403094$ \\
$n_i=100$, $z=0$ \hspace{4cm} $n_i=100$, $z=0.01$ %$n=100$ \hspace{3cm} $n=100.01$ 
\caption{Same as Figure \ref{fig:QR1} except now $k=2.19476917403094$, which is also (an approximation to) a quasi-resonance when $n_i=100$.}
\label{fig:QR2}
\end{figure}

There has been sustained interest in the mathematics and physics communities in studying the stability/perturbation of resonances. 
Quantitative results about how the resonances behave under small perturbations of the wave speed and/or domain in specific situations are given in, e.g, 
\cite{Ra:80}
\cite[\S3]{AlHo:84}, \cite{HoSi:96, %MaSaYuMa:07, 
HeBeKoWe:08, Mc:17, AmDaFiMi:20, AmDaFiMi:21},
and rather general results about resonance stability and simplicity under perturbation are given by, e.g., 
\cite{St:94, KlZw:95, AmTr:04, Sj:14, Xi:23}.
Theorems \ref{thm:main} and \ref{thm:mainT} give a new perspective, complementary to these existing ones, on the (in)stability of resonances under perturbations of the wave speed; indeed, here we are concerned not just with the location of resonances, but with the solution-operator norm in their vicinity.

The next subsection describes how Theorems \ref{thm:main} and \ref{thm:mainT} are proved using techniques originally introduced to show that existence of quasimodes with superalgebraically small error implies existence of resonances superalgebraically close to the real axis \cite{StVo:95, StVo:96, TaZw:98}.

\subsection{The ideas behind the proofs of Theorems \ref{thm:main} and \ref{thm:mainT}}\label{sec:ideas}

Part (b) of  Theorem \ref{thm:main}/\ref{thm:mainT} can be viewed as a counterpart to the solution-operator bound in \cite[Theorem 3.3]{LSW1}. Indeed \cite[Theorem 3.3]{LSW1} showed that the solution operator for a wide variety of scattering problems is polynomially bounded in $k$ for ``most" $k$ \footnote{We note that, under an additional assumption about the location of resonances, a similar result with a larger polynomial power can also be extracted from \cite[Proposition 3]{St:01} by using the Markov inequality.}
. Applied to a problem with quasi-resonances, this result shows that the superalgebraic growth of the solution through $\{k_j\}_{j=1}^\infty$ is unstable with respect to ``most'' perturbations in $k$. 
Part (b) of Theorem \ref{thm:main}/\ref{thm:mainT} shows that this superalgebraic growth is unstable with respect to ``most"  single-signed perturbations of the wave speed.

The proofs of Part (b) of Theorem \ref{thm:main}/\ref{thm:mainT} follow the same outline as the proof of \cite[Theorem 3.3]{LSW1}. Indeed, 
the ingredients of the proof of \cite[Theorem 3.3]{LSW1} are 
\begin{enumerate}
\item the semiclassical maximum principle -- a consequence of the Hadamard three-lines theorem of complex analysis, and originally used in \cite{TaZw:98}, 
\item a bound on the number of poles of $(P-z)^{-1}$ where $P= -k^{-2} n^{-1}\Delta$ -- see \cite[Theorems 4.13 and 7.4]{DyZw:19} and the overview in \cite[\S7.6]{DyZw:19} on the large literature on proving such a bound, 
\item  a bound on $(P-z)^{-1}$, with $z$ a prescribed distance away from poles (coming from \cite{StVo:95, StVo:96, TaZw:98}), and 
\item the bound $\|(P-z)^{-1}\|_{L^2\to L^2}\leq 1/(\Im z)$ for $\Im z>0$, coming, e.g., from considering the pairing $\langle (P-z)u,u\rangle_{L^2}$ and then using self-adjointness of $P$. 
\end{enumerate}

In our setting, however, we need to prove analogues of (2)-(4) with the 
additional complication that the perturbation $z\psi$, unlike a
spectral parameter, is not supported everywhere (since $\psi$ has compact support). 
Note that the analogue of Point (2) is then Part (a) of Theorems \ref{thm:main} and \ref{thm:mainT}.

The assumption in Theorems \ref{thm:main} and \ref{thm:mainT} that $\psi>0$ in a suitable region comes from Ingredient (4) -- this sign condition on $\psi$ ensures there is a half-space in $z$ where one can obtain a bound with $1/(\Im z)$ on the right-hand side (see Lemmas \ref{lem:psi} and \ref{lem:psiT} below).
Indeed, the imaginary part of the pairing $\langle (-k^{-2}\Delta - n -z\psi)u,u\rangle_{L^2}$ gives information about $u$ on the support of $\psi$; we then propagate this information off $\supp \psi$ via a commutator argument. In the proof of 
Theorem \ref{thm:main} this commutator argument occurs in the setting of (semiclassical) defect measures (with our default references \cite{Zw:12, DyZw:19}); see Lemma \ref{lem:psi} below.
In the proof of Theorem \ref{thm:mainT} we commute with $x\cdot \nabla$ (plus lower-order terms); see Lemma \ref{lem:psiT} below. Recall that this commutator was pioneered by Morawetz in the setting of obstacle scattering \cite{MoLu:68, Mo:75}, with the ideas recently transposed to the penetrable obstacle case in \cite{MoSp:19}.

\subsection{Discussion of Theorems \ref{thm:main} and \ref{thm:mainT} in the context of uncertainty quantification}

One motivation for proving Theorems \ref{thm:main} and \ref{thm:mainT} comes from uncertainty quantification (UQ). The forward problem in UQ of PDEs is to compute statistics of quantities of interest involving PDEs \emph{either} posed on a random domain \emph{or} having random coefficients.

A crucial role in UQ theory is understanding regularity of the solution $u$ with respect to $\newy$, where 
$\newy$ is a vector of parameters governing the randomness and the problem is posed 
in the abstract form 
$P(\newy) u(\newy) = f$,
with $P$ a differential or integral operator. 
Indeed, some of the strongest UQ convergence results are obtained by 
proving that $u$ is holomorphic with respect to (the complex extension of) $\newy$,
or by proving equivalent bounds on the derivatives;
see, e.g.,  \cite[Theorem 4.3]{CoDeSc:10}, \cite{CoDeSc:11},  \cite[Section 2.3]{KuSc:13}.
This parametric holomorphy 
allows one to establish
rates of convergence 
for 
stochastic-collocation/sparse-grid
schemes, see, e.g., \cite{ChCoSc:15,CaNoTe:16, HaHaPeSi:18}, quasi-Monte Carlo (QMC) methods, see, e.g., 
\cite{Sc:13, DiKuLeNuSc:14, DiKuLeSc:16, KuNu:16, HaPeSi:16}, Smolyak quadratures, see, e.g., \cite{ZeSc:20}, and deep-neural-network approximations of the solution; 
see, e.g., \cite{ScZe:19, OpScZe:21, LoMiRuSc:21}.

At least for the Dirichlet obstacle problem, existence of super-algebraically small quasimodes for $\chi(-k^{-2}\Delta - 1-i0)^{-1}\chi$ (i.e., \eqref{eq:superalgebraic}) implies the existence of a pole of $z\mapsto \chi (-k^{2}\Delta - 1 -z1_{B(0,R)}-i0)^{-1}\chi$ superalgebraically-close to the origin by \cite[Theorem 1.5]{GaMaSp:21} (see also \cite[Theorem 1.11]{SW1}); we expect the analogous result to be true for the variable-wave-speed problem \eqref{eq:Helmholtz} when $n$ is smooth (indeed, the propagation arguments are simpler in the case when there is no boundary).

Part (a) of Theorem \ref{thm:main} and Part (a) of Theorem \ref{thm:mainT}
immediately imply that, given $k>0$, for ``most'' $z \in (-\epsilon,\epsilon)$ the map $z\mapsto \chi (-k^{-2} \Delta -n-z\psi-i0)^{-1}\chi$ 
is holomorphic in a ball of radius $\sim k^{-M}$; i.e., the bad
behaviour exhibited above by \cite[Theorem 1.5]{GaMaSp:21} is rare. 

\begin{corollary}[Holomorphy of solution operator in $B(z,Ck^{-M})$ for ``most" $z$]\label{cor:holo}

\

(a) Under the assumptions of Theorem \ref{thm:main}, 
given $\smallp, \epsilon,\delta,k_0>0$ and $\chi \in C^{\infty}_{\rm comp}(\Rea^d)$,  
there exists $C>0$ such that, for all $k\geq k_0$, there exists $S_k\subset (-\smallp,\smallp)$ with $|S_k|<\delta$ such that for all $z_0\in (-\smallp,\smallp)\setminus S_k$, the map 
$z\mapsto \chi (-k^{-2} \Delta -n-z\psi)^{-1}\chi$ 
is holomorphic in $B(z_0,Ck^{-(d+1+\epsilon)})$.

(b) Under the assumptions of Theorem \ref{thm:mainT}, 
given $\smallp,\delta,k_0>0$
and $\chi \in C^{\infty}_{\rm comp}(\Rea^d)$,  there exists $C>0$ such that, for all $k\geq k_0$, there exists $S_k\subset (-\smallp,\smallp)$ with $|S_k|<\delta$ such that for all $z_0\in (-\smallp,\smallp)\setminus S_k$, the map 
$z\mapsto \chi (-k^{-2} \Delta -n-z\psi)^{-1}\chi$ 
is holomorphic in $B(z_0,Ck^{-(d+2)})$.
\end{corollary}

\begin{proof}
We prove the result in Part (a); the proof of the result in Part (b) is completely analogous. 
By Theorem \ref{thm:main}, given $\epsilon>0$, in a $k$-independent neighbourhood of $(-\smallp,\smallp)$ there are at most $C_1 k^{d+1+\epsilon}$ poles $\mathcal{P}$. Let $C := \delta/(2C_1)$, and let
\[S_k:= \bigcup_{p\in\mathcal{P}}(-\smallp,\smallp)\cap B(p,Ck^{-(d-1-\epsilon)}).\]
Then, since $|(-\smallp,\smallp)\cap B(p,Ck^{-(d-1-\epsilon)})|\le 2Ck^{-(d-1-\epsilon)}\le\delta C_1^{-1}k^{-(d-1-\epsilon)}$ for any $p$, we have $|S_k|\le |\mathcal{P}|\delta C_1^{-1}k^{-(d-1-\epsilon)}\le \delta$; furthermore, by definition, if $z_0\in(-\smallp,\smallp)\backslash S_k$, then $B(z_0,Ck^{-(d-1-\epsilon)})$ does not contain a pole.
\end{proof}

\subsection{Outline of the paper}
Section \ref{sec:mero} contains 
results about meromorphic continuation and complex scaling.
Section \ref{sec:poly_bound} proves Part (a) of Theorems \ref{thm:main} and \ref{thm:mainT}
(a polynomial bound on the number of poles).
Section \ref{sec:resolvent} proves Part (b) of Theorems \ref{thm:main} and \ref{thm:mainT} (a bound on the solution operator for real $z$). 
Section \ref{app:A} recaps relevant results from semiclassical analysis.
Section \ref{app:B} recaps relevant results about Fredholm and trace-class operators.

\section{Results about meromorphic continuation and complex scaling}\label{sec:mero}

We replace the
large spectral parameter $k$ by a small parameter
$$
h:= k^{-1},
$$
and let 
\beq\label{eq:P}
P:= -h^2 \Delta -n = -h^2\Delta +(1-n)-1;
\eeq
intuitively, then, we are thinking about the semiclassical
Schr\"odinger operator with compactly supported potential $1-n$, considered at energy $1$.

Note that for $z \in \RR$, $P-z\psi$ is essentially self-adjoint,
hence for $\ep>0$,
$$
(P-z\psi-i\ep )^{-1}: L^2(\Rea^d) \to H^2(\Rea^d).
$$
\begin{lemma}[Limiting absorption principle]
\label{lem:allgood}
Let $\psi \in L^\infty_{\rm comp}(\mathbb{R}^d;\mathbb{R})$. 
For $z\in \mathbb{R}$, 
$$(P-z\psi-i0)^{-1}:= \lim_{\ep\downarrow 0}
(P-z\psi-i\ep )^{-1} : L^2_{\rm comp}(\mathbb{R}^d)\to H^2_{\rm
  loc}(\mathbb{R}^d)$$ and, given $f\in L^2_{\rm comp}(\mathbb{R}^d)$,  
$u:=(P-z\psi-i0)^{-1}f$ is the unique solution to $(P-z\psi)u=f$ satisfying the Sommerfeld radiation condition \eqref{eq:src}.
\end{lemma}
  
\begin{proof}[References for the proof]
The existence of $(P-z\psi-i0)^{-1}$ follows from \cite[Theorem
3.8]{DyZw:19} (in odd dimension) and \cite[Theorem 4.4]{DyZw:19} (for
general dimension). Uniqueness and the radiation condition follow from
\cite[Theorems 3.33 and 3.37]{DyZw:19} (noting that these results
don't require the dimension to be odd---see \cite[p.251]{DyZw:19} for
remarks on this extension of results for odd dimensional potential
scattering to the ``black box'' setting in arbitrary dimension). 
The convergence as $\ep\downarrow 0$ holds in the strong operator topology, which has
  a more precise global statement in terms of appropriately
  \emph{weighted} Hilbert spaces; see, e.g., \cite[Proposition
  1.2]{RoTa:14} for a careful discussion.
\end{proof}

We now show meromorphic continuation. For $\chi\in C^\infty_{\rm comp}(\mathbb{R}^d)$ and $z\in\mathbb{R}$, let
\[\resolve_\chi(z) = (P-z\psi-i0)^{-1}\chi: L^2(\mathbb{R}^d)\to H^2_{\rm loc}(\mathbb{R}^d).\]
\begin{lemma}[Meromorphic continuation]
\label{lem:meromorphic}
Let $\psi \in L^\infty_{\rm comp}(\mathbb{R}^d;\mathbb{R})$. 

(i) For all
$\chi \in C^\infty_{\rm comp}(\mathbb{R}^d)$,
$\resolve_\chi(z)$ extends from $z \in\RR$ to a 
meromorphic family of operators $L^2(\mathbb{R}^d)\to H^2_{\rm loc}(\mathbb{R}^d)$ for $z\in \CC$. Moreover, for $z\in\CC$ that is not a pole of $\resolve_\chi(z)$ and any $f\in L^2(\mathbb{R}^d)$, the image $\resolve_\chi(z)f$ satisfies the Sommerfeld radiation condition \eqref{eq:src}.

(ii) If $\chi \equiv 1$ on  $\supp(1-n)\cup\supp(\psi)$, then the
poles of $\resolve_\chi(z)$ do not depend on the particular choice of $\chi$. Furthermore, 
with $\mathcal{P}$ denoting the poles of $\resolve_\chi(z)$ for any such
choice of $\chi$, 
then for any $\widetilde\chi\in C^\infty_{\rm
  comp}(\mathbb{R}^d)$ the poles of $\resolve_{\widetilde{\chi}}(z)$ are contained in $\mathcal{P}$.

(iii) For any $\chi_1,\chi_2\in C^\infty_{\rm comp}(\mathbb{R}^d)$ and
any $z\in\CC\backslash\mathcal{P}$, $\resolve_{\chi_1}(z)\chi_2 = \resolve_{\chi_2}(z)\chi_1$.

\end{lemma}
Consequently, for $z\in\CC\backslash\mathcal{P}$, we can define the operator
\[(P-z\psi-i0)^{-1}: L^2_{\rm comp}(\mathbb{R}^d)\to H^2_{\rm loc}(\mathbb{R}^d)\]
by
\begin{equation}
\label{eq:resolvent-def}
(P-z\psi-i0)^{-1}f := \resolve_{\chi}(z)f
\end{equation}
for $f\in L^2_{\rm comp}(\mathbb{R}^d)$, where $\chi\in C^\infty_{\rm comp}(\mathbb{R}^d)$ is any function that is identically equal to $1$ on $\supp(1-n)\cup\supp(\psi)\cup\supp(f)$. This definition agrees with the definition in Lemma \ref{lem:allgood} for $z\in\RR$, and moreover part (iii) of Lemma \ref{lem:meromorphic} shows the definition \eqref{eq:resolvent-def} does not depend on the choice of $\chi$.

With $(P-z\psi-i0)^{-1}$  defined as in \eqref{eq:resolvent-def}, we check that it gives the outgoing solution to $(P-z\psi)u=f$:

\begin{lemma}\label{lem:joey}
Let $z\in\CC\backslash\mathcal{P}$ and $f\in L^2_{\rm comp}(\mathbb{R}^d)$, and let $u=(P-z\psi-i0)^{-1}f$. Then ${(P-z\psi)u} = f$, and $u$ satisfies the Sommerfeld radiation condition \eqref{eq:src}. 
\end{lemma}

\begin{proof}[Proof of Lemma \ref{lem:meromorphic}]
(i) We use perturbation arguments, since our operator $P$ is a
relatively compact
perturbation of the free Helmholtz operator $P_0 = -h^2\Delta-I$.
Since 
\beqs
(P-z\psi) (P_0-i0)^{-1} = I - (n-1+z\psi)(P_0-i0)^{-1},
\eeqs
we have 
\beq\label{eq:Sat2}
(P-z\psi-i0)^{-1}= (P_0-i0)^{-1}\big(I - (n-1+z\psi) (P_0-i0)^{-1}\big)^{-1}.
\eeq
Given $\chi  \in C^\infty_{\rm comp}(\mathbb{R}^d)$ with $\chi \equiv 1$ on  $\supp(1-n)\cup\supp(\psi)$, a standard argument (see \cite[Equation 3.2.2 and proof of Theorem 2.2]{DyZw:19}\footnote{
Note that \cite[Equation 3.2.2]{DyZw:19},  although appearing in a section on
odd-dimensional scattering, is valid in all dimensions.} or \cite[Page 6739]{GaMaSp:21}),
then allows one to insert factors of $\chi$ into \eqref{eq:Sat2}; indeed, 
 for $z \in \RR$,
\begin{equation}\label{factorization}
 (P-z\psi-i0)^{-1}\chi =(P_0-i0)^{-1}\chi ( I-(n-1+z \psi) (P_0-i0)^{-1} \chi )^{-1}.
\end{equation}
Let
\beqs
%\label{eq:kchi}
 K_\chi(z):=(n-1+z \psi) (P_0-i0)^{-1} \chi
\eeqs
so that 
$$ \resolve_\chi(z) = (P-z\psi-i0)^{-1}\chi =(P_0-i0)^{-1}\chi ( I-K_\chi(z) )^{-1}.$$
Since $ (P_0-i0)^{-1}: L^2_{\rm comp}\to H^2_{\rm loc}$ and $\supp (n-1 +z\psi)$ is compact,
the operator family $K_\chi(z)$ 
  is
a family of compact operators, holomorphic in $z$.
If we can show that $I-K_\chi(z)$ is invertible for some $z\in \mathbb{C}$, then the analytic Fredholm theorem (see,
Theorem \ref{thm:analytic_Fredholm} below)
  implies that $(I-K_\chi(z))^{-1}$ is meromorphic in $z$ as a family of bounded operators $L^2(\mathbb{R}^d)\to L^2(\mathbb{R}^d)$. Thus 
\beq\label{eq:Sat5}
\resolve_\chi(z) =(P_0-i0)^{-1}\chi (I-K_\chi(z))^{-1}
\eeq
is a meromorphic family of operators $L^2(\mathbb{R}^d)\to H^2_{\rm loc}(\mathbb{R}^d)$. Moreover, for any $f\in L^2(\mathbb{R}^d)$ and any $z$ that is not a pole of $\resolve_\chi$, we have
\[\resolve_\chi(z)f = (P_0-i0)^{-1}g,\quad g = \chi(I-K_\chi(z))^{-1}f\in L^2_{\rm comp}(\mathbb{R}^d),\]
from which we have that $\resolve_\chi(z)$ satisfies the Sommerfeld radiation conditions \eqref{eq:src} due to the mapping properties of $(P_0-i0)^{-1}$.

We now prove that $I-K_\chi(z)$ is invertible for $z=0$. Since 
\beqs
K_\chi(z) = (n-1+z\psi)(P_0-i0)^{-1}\chi,
\eeqs
it follows that $u\in\ker(I-K_\chi(z))$ if and only if
\beq\label{eq:Sat3}
u = (n-1+z\psi)(P_0-i0)^{-1}\chi u.
\eeq
For such $u$, $\supp(u)\subseteq\supp(n-1+z\psi)\subseteq\supp(1-n)\cup\supp(\psi)$, and since $\chi \equiv 1$ on $\supp(1-n)\cup\supp(\psi)$, we have $\chi u=u$.
Thus, for $u\in\ker(I-K_\chi(0))$, we have 
\beqs
u = (n-1) (P_0-i0)^{-1}u. 
\eeqs
Let $v:= (P_0-i0)^{-1}u$ so that $u=(n-1)v$ and 
\beqs
P_0 v = u = (n-1)v;  \quad \text{ i.e., } (-h^2 \Delta - n)v=0.
\eeqs
By its definition, $v$ is outgoing, and thus and $v\equiv 0$ by Rellich's uniqueness theorem (see, e.g., \cite[Theorem 3.33]{DyZw:19}). 
Therefore $u = (n-1)v \equiv 0$ as well.

We have therefore proved that 
$\resolve_\chi(z) = (P-z\psi-i0)^{-1}\chi$
extends meromorphically from $\RR$ to $\CC$ when  $\chi \equiv 1$ on  $\supp(1-n)\cup\supp(\psi)$. The fact that $\resolve_\chi(z) = (P-z\psi-i0)^{-1}\chi$ extends meromorphically for general $\chi\in C_{\rm comp}^\infty(\mathbb{R}^d)$ follows by 
noting for $z\in\RR$ that $\resolve_\chi(z) = \resolve_{\widetilde{\chi}}(z)\chi$
where $\widetilde\chi\in C_{\rm comp}^\infty(\mathbb{R}^d)$ is identically equal to $1$ on $\supp(\chi)\cup\supp(\psi)\cup\supp(1-n)$
and applying the meromorphic extension of $\resolve_{\widetilde\chi}$ obtained above.

(ii) By \eqref{eq:Sat5} and \eqref{eq:Sat3}, 
if $\chi \equiv 1$ on $\supp(1-n)\cup\supp(\psi)$ then 
$z$ is a pole of $\resolve_\chi(z)$ if and only if there
exists a nonzero $u\in L^2(\mathbb{R}^d)$, with
$\supp(u)\subseteq\supp(1-n)\cup\supp(\psi)$, satisfying
\[u = (n-1+z\psi)(P_0-i0)^{-1}u.\]
Since this condition is independent of $\chi$,  the poles of
$\resolve_\chi(z)$ do not depend on the choice of $\chi$,
as long as it equals $1$ on $\supp(1-n)\cup\supp(\psi)$.

For general $\widetilde\chi\in C^\infty_{\rm comp}(\mathbb{R}^d)$, $\resolve_{\widetilde{\chi}}(z) = \resolve_{\chi}(z)\widetilde\chi$ where $\chi\in C^\infty_{\rm comp}(\mathbb{R}^d)$ is identically equal to $1$ on $\supp(1-n)\cup\supp(\psi)\cup\supp(\widetilde\chi)$, from which we see that all poles of $\resolve_{\widetilde\chi}$ are contained in $\mathcal{P}$.

(iii)
The relation $\resolve_{\chi_1}(z)\chi_2 = \resolve_{\chi_2}(z)\chi_1$ holds when $z\in\mathbb{R}$, and hence by analytic continuation it continues to hold for $z\in\CC\backslash\mathcal{P}$ as well.
\end{proof}

\begin{proof}[Proof of Lemma \ref{lem:joey}]
By the definition of $(P-z\psi-i0)^{-1}f$ in \eqref{eq:resolvent-def}, 
${(P-z\psi-i0)^{-1}f}$ $= \resolve_\chi(z)f$ where $\chi$ is identically equal to $1$ on $\supp(1-n)\cup\supp(\psi)\cup\supp(f)$, from which we obtain the Sommerfeld radiation condition by the mapping properties of $\resolve_\chi(z)$ proved in Lemma \ref{lem:meromorphic}, Part (a). Moreover, the equation
\[(P-z\psi)\resolve_\chi(z)f = (P-z\psi)(P-z\psi-i0)^{-1}\chi f = \chi f = f\]
holds when $z\in\RR$, and hence by analytic continuation it continues to hold for $z\in\CC\backslash\mathcal{P}$, thus showing that $(P-z\psi)(P-z\psi-i0)^{-1}f = f$ continues to hold for $z\in\CC\backslash\mathcal{P}$ as well.
\end{proof}

We now define a special case of \emph{complex scaling} (for the general case, see, e.g., \cite[\S4.5.1]{DyZw:19}). 
Recall that $R_0>0$ is such that $\supp(1-n)\cup\supp(\psi)\subset B(0,R_0)$. Given $R_1 >R_0$ and $R_2>2R_1$, 
let $g\in C^\infty(\Rea;\Rea)$ be such that 
\beq\label{eq:g}
g(t) = 0 \,\,\text{  for  }\,\, t\leq R_1, \quad g(t) = t^2/2 \,\,\text{  for  } \,\,t\geq R_2, \quad\text{  and  }\,\, g''(t)\geq 0\,\, \text{ for all }t.
\eeq
Then let
\beq\label{eq:F}
F_\theta (x) = (\tan \theta) g(|x|) \quad \text{ for }\quad 0<\theta<\pi/2,
\eeq
and define 
\beq\label{eq:DeltaTheta}
-\Delta_\theta u := \big( (I+i F''_\theta(x))^{-1}\partial_x \big) \cdot  \big( (I+i F''_\theta(x))^{-1}\partial_x u\big),
\eeq
where $F''_\theta$ denotes the Hessian of $F_\theta$; see \cite[Equation 4.5.14]{DyZw:19}. Observe that  $\Lap_\theta = \Lap$ on $B(0,R_1)$ and $\Lap_\theta = (1+i\tan\theta)^{-2}\Lap$ outside $B(0,R_2)$. 
Finally, let 
\beq\label{eq:scaled}
P_\theta=-h^2\Lap_\theta -n.
\eeq

\begin{lemma}\label{lem:Fredholm}
$P_\theta -z\psi$ is Fredholm of index zero and the poles of $z\mapsto(P_\theta-z\psi)^{-1}$ are discrete.
\end{lemma}

\begin{proof}
\cite[Lemma 4.36]{DyZw:19} implies that $P_\theta$ is Fredholm of
index zero, mapping $H^2_h\to L^2 $ (where the semiclassical Sobolev space $H^2_h$ is defined by \eqref{eq:Hsk} below), whenever $P$ is a semiclassical
black-box operator in the sense of \cite[Definition 4.1]{DyZw:19};
this is the case for $P$ defined by \eqref{eq:P} (in
particular, $P$ equals $-h^2 \Delta-I$ outside a compact set).
Moreover, unique continuation 
(which holds when $n\in L^\infty$ by \cite{JeKe:85}) and  \cite[Theorems~4.18 and
4.38]{DyZw:19} show
$P_\theta$ has bounded inverse as a map from $H^2_h$ to $L^2$.
Consequently, since multiplication by $\psi$ is a compact
operator
  $H^2_h$ to $L^2$, the Fredholm alternative implies that the factorization
\begin{equation}\label{Pthetafactor}
(P_\theta-z\psi)=  P_\theta(I+ (P_\theta)^{-1} z\psi),
\end{equation}
exhibits $(P_\theta-z\psi)$ as an invertible operator $H^2_h\to L^2$ right-composed with
a holomorphic family of operators on $H^2_h$ of index zero; the analytic Fredholm
theorem (see Theorem \ref{thm:analytic_Fredholm} below) then yields discreteness of the poles of the inverse.
\end{proof}

\begin{lemma}[Agreement of the resolvents away from scaling]\label{lem:agreement}
If $\chi \in C^{\infty}_{\rm comp}(B(0, R_1))$, then
\beq\label{eq:agreement}
\chi(P_\theta-z\psi)^{-1}\chi = \chi(P-z\psi-i0)^{-1}\chi
\eeq
whenever $z$ is not a pole of $(P_\theta-z\psi)^{-1}$.
\end{lemma}

\begin{proof}
We first suppose that $\chi$ is identically equal to one near $B(0,R_0)$. 
When $z\in \mathbb{R}$, 
 $\widetilde{P}:= P -z\psi$ is semiclassical black-box operator (in the sense of \cite[Definition 4.1]{DyZw:19}) since $\supp(\psi)\subset B(0,R_0)$.
The agreement \eqref{eq:agreement} for $z\in \mathbb{R}$ then follows from \cite[Theorem 4.37]{DyZw:19} (with $\lambda=1$).
Since both sides of \eqref{eq:agreement} are meromorphic in $z$
(by Lemmas \ref{lem:Fredholm} and \ref{lem:meromorphic}, respectively), \eqref{eq:agreement} holds for 
all $z\in \mathbb{C}$ that are not poles of $(P_\theta-z\psi)^{-1}$
by analytic continuation.

For a general $\chi\in C^{\infty}_{\rm comp}(B(0,R_1))$, there exists $\widetilde\chi\in C^{\infty}_{\rm comp}(B(0,R_1))$ that is identically one on $B(0,R_0)\cup\supp(\chi)$, in which case $\chi = \chi\widetilde\chi$. Then the previous paragraph gives $\widetilde\chi(P_\theta-z\psi)^{-1}\widetilde\chi = \widetilde\chi(P-z\psi-i0)^{-1}\widetilde\chi$, after which multiplying on the left and right by $\chi$ gives the desired agreement.
\end{proof}

%Given $\chi \in C^\infty_{\rm comp}(\Rea^d)$, without loss of generality we can assume that $\supp \chi \subset B(0,R_1)$ (since we can always increase $R_1$, i.e., the point at which the scaling starts).
%If $\chi(P_\theta-z\psi)^{-1}\chi$ has a pole at $z=z_0$ then 
%$(P_\theta-z\psi)^{-1}$ has a pole at $z=z_0$; therefore  
%%$\chi(P_\theta-z\psi)^{-1}\chi$ has a pole at $z=z_0$ only if 
%%the same is true for 
%%$(P_\theta-z\psi)^{-1}$, 
%an immediate corollary is the following.

%Noting that $\chi(P_\theta-z\psi)^{-1}\chi$ has a pole at $z=z_0$ only if 
%the same is true for 
%$(P_\theta-z\psi)^{-1}$, we get as an immediate corollary the following.
%\jzmar{We never said what $\chi$ is in this corollary. I think this should hold for all compactly supported $\chi$--it immediately holds for $\chi$ supported in $B(0,R_1)$; presumably in general it could follow from Lemma \ref{lem:meromorphic}, part (ii)?}
\begin{corollary}\label{cor:poles}
Given $\chi \in C^\infty_{\rm comp}(\Rea^d)$, choose $R_1>R_0$ large enough so that $\supp \chi \subset B(0,R_1)$. Let $P_\theta$ be the complex scaled operator \eqref{eq:scaled} with this $R_1$. Then 
%For $\chi \in C^{\infty}_{\rm comp}(B(0, R_1))$, 
the number of poles of $\chi(P-z\psi)^{-1}\chi$ is at most the number of poles of $(P_\theta-z\psi)^{-1}$.
\end{corollary}

\begin{proof}
%Given $\chi \in C^\infty_{\rm comp}(\Rea^d)$, let $R_1>0$ be large enough so that $\supp \chi \subset B(0,R_1)$. 
If $\chi(P_\theta-z\psi)^{-1}\chi$ has a pole at $z=z_0$ then 
$(P_\theta-z\psi)^{-1}$ has a pole at $z=z_0$; the result then follows from Lemma \ref{lem:agreement}.
%$\chi(P_\theta-z\psi)^{-1}\chi$ has a pole at $z=z_0$ only if 
%the same is true for 
%$(P_\theta-z\psi)^{-1}$, 
%an immediate corollary is the following.
\end{proof}

Thus, to count poles of $\chi (P-z\psi-i0)^{-1}
    \chi$ for $\chi \in C^\infty_{\rm comp}(\Rea^d)$, it suffices to count the number of poles of $(P_\theta-z\psi)^{-1}$.

\section{Polynomial bound on the number of poles (proof of Part (a) of Theorems \ref{thm:main} and \ref{thm:mainT}).}\label{sec:poly_bound}

Given $\smallp>0$, let $\Omega \Subset \mathbb{C}$ be such that $\Omega\supset \{z : |z|< \smallp\}.$ 
The bounds in Part (a) of Theorems \ref{thm:main} and \ref{thm:mainT} follow from a bound on the number of poles in $\Omega$.

By Corollary \ref{cor:poles}, it is sufficient to bound the number of poles of $(P_\theta-z\psi)^{-1}$. As noted in \S\ref{sec:ideas}, we follow the steps in the proof of the analogous bound on the number of poles of  $(P- zI)^{-1}$ in \cite[Theorem 7.4]{DyZw:19}.
Let 
$$
\widetilde P_\theta:= P_\theta -iM Q
$$
where
$$
 Q :=  \chi(hD) \chi^2(x) \chi(hD),
 $$
$M>0$ is sufficiently large, and $\chi\in C^\infty_{\rm comp}(\mathbb{R}^d)$ is identically one on $B(0,R_3)$, where $R_3>R_2$ is sufficiently large. Both $M$ and $R_3$ will be specified later. 
Set
\beq\label{eq:W}
W:=\widetilde P_\theta - z\psi,
\eeq
%The following result uses the $h$-weighted space $H^2_h(\Rea^d)$ and norm $\|\cdot\|_{H^2_h(\Rea^d)}$ defined below in \eqref{eq:Hsk} and \eqref{eq:Hhnorm}, respectively.

 \begin{lemma}[$W$ is invertible, uniformly for $h$ sufficiently small]\label{lem:W}
Given $h_0>0$ and $\Omega \Subset \mathbb{C}$, if $M$ and $R_3$ are sufficiently large then there exists $C>0$ such that, for 
all $0<h<h_0$ and $z\in \Omega$, $W^{-1}: L^2(\Rea^d) \to H^2_h(\Rea^d)$ exists with %and there exists $C>0$ such that 
$\|W^{-1}\|_{L^2\to H^2_h} \leq C$.
\end{lemma}

\begin{proof}%[Proof of Lemma \ref{lem:W}]
Step 1:~We claim that if $R_3$ is sufficiently large, then $\widetilde P_\theta:H^2_h\to L^2$ is invertible for all sufficiently small $h$, with $\|\widetilde P_\theta^{-1}\|_{L^2\to H^2_h}$ uniform in $h$. This will follow from semiclassical elliptic regularity, i.e.\ from showing a bound of the form
\[|\sigma_h(\widetilde P_\theta)(x,\theta)|\ge \epsilon(1+|\xi|^2);\]
see Theorem \ref{thm:elliptic} below. 
From \eqref{eq:DeltaTheta}/\cite[Equation 4.5.14]{DyZw:19},  
\[\sigma_h(-\Lap_\theta) = |\eta|^2 - |F''_\theta(x)\eta|^2 - 2i\langle F''_\theta(x)\eta,\eta\rangle,\quad \eta = (1+(F''_\theta(x))^2)^{-1}\xi.\]
Recall from \eqref{eq:g} and \eqref{eq:F} that $F''_\theta(x)$ is always positive semi-definite, $F''_\theta(x) = \tan(\theta)I$ for $|x|\ge R_2$, and we can choose $\theta$ sufficiently small so that $|F''_\theta(x)\eta|\le\frac{1}{2}\eta$ for all $\eta$; note that this implies $|\xi| = |(1+(F''_\theta(x))^2)\eta|\le \frac{5}{4}|\eta|$, i.e.\ $|\eta|\ge\frac{4}{5}|\xi|$. Recall that
\[\widetilde P_\theta := P_\theta - iMQ = -h^2\Lap_\theta - n(x) - iM\chi(hD)\chi(x)^2\chi(hD)\]
so that
\[\sigma_h(\widetilde P_\theta) = |\eta|^2 - |F''_\theta(x)\eta|^2 - n(x) - i\left(2\langle F''_\theta(x)\eta,\eta\rangle + M\chi(x)^2\chi(\xi)^2\right).\]
From this, we easily see that $|\sigma_h(\widetilde P_\theta)|\ge \epsilon(1+|\xi|^2)$ for $\xi$ sufficiently large, so it suffices to show it is uniformly bounded away from $0$ for $\xi$ uniformly bounded. First, we claim $\sigma_h(\widetilde P_\theta)$ never vanishes:~indeed, since $F''_\theta$ is positive semi-definite, 
\[\Im(\sigma_h(\widetilde P_\theta)) = -\left(2\langle F''_\theta(x)\eta,\eta\rangle + M\chi(x)^2\chi(\xi)^2\right)\le 0;\]
moreover, the last inequality holds with equality if and only if
\beq\label{eq:spots1}
F''_\theta(x)\eta = 0\text{ and }\chi(x)\chi(\xi) = 0.
\eeq
If $F''_\theta(x)\eta = 0$, then
\[\Re(\sigma_h(\widetilde P_\theta)) = |\eta|^2-|F''_\theta(x)\eta|^2-n(x) = |\eta|^2-n(x),\]
so $\sigma_h(\widetilde P_\theta)(x,\xi) = 0$ implies that $|\eta|^2= n(x)$, which implies that $|\xi|^2\le \frac{25}{16}n(x)$. As such, if we arrange
\[R_3^2>\frac{25}{16}\sup_{\mathbb{R}^d} n(x),\]
then, since $\chi \equiv 1$ on $B(0,R_3)$, 
$\sigma_h(\widetilde P_\theta)(x,\xi) = 0$ implies that $\chi(\xi) =
1$, which, by the second inequality in \eqref{eq:spots1}, implies that $\chi(x) = 0$. But this forces $|x|>R_3$ since $\chi$ is
identically one on $B(0,R_3)$; in particular this means $F''_\theta(x)
= \tan(\theta)\text{Id}$, and hence $F''_\theta(x)\eta = 0$ which implies that
$\eta = 0$, contradicting $|\eta|^2=n(x)>0$. This argument establishes
that $\sigma_h(\widetilde P_\theta)$ is never zero, with the uniform
bound away from zero following by the ellipticity for $\xi$ large and
by noting that $\sigma_h(\widetilde P_\theta)$ does not depend on $x$
for $|x|$ sufficiently large. We have therefore established that
$\widetilde P_\theta:H^2_h\to L^2$ is invertible for all sufficiently small $h$, with $\|\widetilde P_\theta^{-1}\|_{L^2\to H^2_h}$ uniform in $h$.

Step 2:~We now claim that if  $M$ and $R_3$ are sufficiently large, then 
\beq\label{eq:spots4}
\|z\psi\widetilde P_\theta^{-1}\|_{L^2\to L^2}\le\frac{1}{2}
\eeq
for all $h$ sufficiently small and all $z\in\Omega$. Indeed, note that if $\widetilde\chi\in C^\infty_{\rm comp}(B(0,R_1))$ is identically one on the support of $\psi$ and $\|\widetilde\chi\|_{L^\infty}\le 1$, then $z\psi\widetilde P_\theta^{-1} = (z\psi)(\widetilde\chi\widetilde P_\theta^{-1})$, and hence
\beq\label{eq:spots3}
\|z\psi\widetilde P_\theta^{-1}\|_{L^2\to L^2}\le \|z\psi\|_{L^2\to L^2}\|\widetilde\chi\widetilde P_\theta^{-1}\|_{L^2\to L^2}\le \left(\sup_{\Omega}|z|\right)\|\psi\|_{L^\infty}\|\widetilde\chi\widetilde P_\theta^{-1}\|_{L^2\to L^2}.
\eeq
Next, $\widetilde\chi\widetilde P_\theta^{-1}\in \Psi_h^{-2}$, with
\beq\label{eq:spots2}
\sigma_h(\widetilde\chi\widetilde P_\theta^{-1})(x,\xi)= \frac{\widetilde\chi(x)}{\sigma_h(\widetilde P_\theta)(x,\xi)} = \frac{\widetilde\chi(x)}{|\xi|^2-n(x)-iM\chi(\xi)^2},
\eeq
where the last equality follows since $\widetilde\chi(x)$ is supported in $B(0,R_1)$, and, for $x\in B(0,R_1)$, $\Lap_\theta = \Lap$ and $\chi(x) = 1$. 
For $x\in B(0,R_1)$, 
\[\left|\Im\left(|\xi|^2-n(x)-iM\chi(\xi)^2\right)\right| = M\chi(\xi)^2 \ge M\text{ if }|\xi|\le R_3\]
and
\[\left|\Re\left(|\xi|^2-n(x)-iM\chi(\xi)^2\right)\right| = ||\xi|^2-n(x)| \ge R_3^2-n(x)\text{ if }|\xi|\ge R_3\]
as long as $R_3^2-n(x)>0$. Hence, if $R_3$ is chosen large enough (depending on $M$) so that $R_3^2-\sup_{\mathbb{R}^d}n(x) \ge M$, then 
\[\left||\xi|^2-n(x)-iM\chi(\xi)^2\right|\ge M\text{ for all }\xi\]
if $x\in B(0,R_1)$, and  thus, by \eqref{eq:spots2} and the fact that $\|\widetilde\chi\|_{L^\infty}\le 1$,
\[|\sigma_h(\widetilde\chi\widetilde P_\theta^{-1})|\le \frac{1}{M}.\]
By Lemma \ref{lem:norm_symbol} below, 
\[\|\widetilde\chi\widetilde P_\theta^{-1}\|_{L^2\to L^2} \le \frac{1}{M}+O(h).\]
Combining this with \eqref{eq:spots3}, we see that if 
$M >
2\left(\sup_{\Omega}|z|\right)\|\psi\|_{L^\infty}$ then 
\[\|z\psi P_\theta^{-1}\|_{L^2\to L^2} \le \left(\sup_{\Omega}|z|\right)\|\psi\|_{L^\infty}\left(\frac{1}{M}+O(h)\right) \le \frac{1}{2}\]
for all sufficiently small $h$. We have therefore established \eqref{eq:spots4} and Step 2 is complete.

We now complete the proof of the lemma. Step 2 implies that $I-z\psi P_\theta^{-1}$ is invertible for all $z\in\Omega$, with $\|I-z\psi P_\theta^{-1}\|_{L^2\to L^2}\le 2$.
By the definition of $W$ \eqref{eq:W},
\[W = \widetilde P_\theta - z\psi = (I-z\psi\widetilde P_\theta^{-1})\widetilde P_\theta,\]
so that $W:H^2_h\to L^2$ is invertible, with
\begin{align*}
W^{-1} &= \widetilde P_\theta^{-1}(I-z\psi\widetilde P_\theta^{-1})^{-1} 
\end{align*}
and
\begin{align*}
\|W^{-1}\|_{L^2\to H^2_h}&\le \|\widetilde P_\theta^{-1}\|_{L^2\to H^2_h}\|I-z\psi P_\theta^{-1}\|_{L^2\to L^2} \le 2\|\widetilde P_\theta^{-1}\|_{L^2\to H^2_h}
\end{align*}
for all $z\in\Omega$.
\end{proof}

As a consequence of Lemma~\ref{lem:W},
\begin{equation}\label{eq:PW}
P_\theta-z \psi  =W + iMQ 
= W(I + W^{-1} i M Q).
\end{equation}
Now let
\beq\label{eq:K}
K(z):=W^{-1} i M Q
\eeq
so that, by \eqref{eq:PW}, 
\begin{equation}\label{eq:PWK}
P_\theta-z \psi  =W(I +K(z))
\end{equation}
By Lemma \ref{lem:W}, $P_\theta -z\psi$ is not invertible iff $I+K(z)$ is not invertible.
Observe that $K(z)$ is compact because $Q$ compact and $W^{-1}$ bounded; therefore, by Part (i) of Theorem \ref{thm:trace}, 
 $P_\theta -z\psi$ is not invertible iff $\det (I +K(z))=0$.

Recall that our goal is to count the number of poles 
of $(P_\theta-z\psi)^{-1}$ 
in $\Omega$; this is equivalent to counting the number of zeros
of the holomorphic function $z\mapsto \det(I+K(z))$ in $\Omega$. If we let $m_K(z)$ denote the order of a zero of $\det(I+K(z))$ (with $m_K(z)=0$ if $z$ is not a zero), then
\[\big(\text{the number of zeros of }\det(I+K(z))\text{ in }\Omega\big) \le \sum_{z\in\Omega} m_K(z).\]
On the other hand,
by Jensen's formula (see, e.g., \cite[Equation D.1.11]{DyZw:19}, \cite[\S3.61]{Ti:39}),
\begin{equation}\label{eq:Jensen}
\sum_{z\in \Omega} 
m_K(z) 
\leq C \sup_{z\in \Omega'} \log |\det(I+ K(z))| -  C \log  |\det(I+ K(z_0))|
\end{equation}
where $\Omega'\supset \Omega$ and $z_0\in \Omega'$.

It thus suffices to estimate the right-hand side. 
Arguing exactly as in \cite[Equation 7.2.8]{DyZw:19}, since the trace-class norm of $Q$, $\|Q\|_{\mathcal{L}_1}$, (defined in Definition \ref{def:trace_class}) equals the trace of $Q$ as $Q\geq 0$, we obtain
\beq\label{eq:boundQ}
\| Q\|_{\mathcal{L}_1}\leq C h^{-d}.
\eeq
Therefore, by the definition of $K(z)$ \eqref{eq:K}
and the composition property of the trace class norm \eqref{eq:composition_trace}, 
$K(z) \in \mathcal{L}_1$ with 
\beq\label{eq:boundK}
\|K(z)\|_{\mathcal{L}_1}\leq M \| W^{-1} \|_{L^2\to L^2} \| Q\|_{\mathcal{L}_1}.
\eeq
Since $K(z)\in \mathcal{L}_1$, Part (ii) of Theorem \ref{thm:trace} implies that  
\beq\label{eq:B511}
\log |\det (I+K(z))|\leq \|K(z)\|_{\mathcal{L}_1}.
\eeq
Combining \eqref{eq:B511}, \eqref{eq:boundK}, \eqref{eq:boundQ} and 
Lemma \ref{lem:W}, we obtain that 
\begin{equation}\label{eq:Jensen1}
\log |\det (I+K(z))| \leq C h^{-d}.
\end{equation}
To obtain a  lower bound on $ \log  |\det(I+ K(z_0))|$ for some $z_0\in \Omega$, we begin by observing that, 
by \eqref{eq:PWK},
\begin{align}\nonumber
\big(I+K(z)\big)^{-1}&=(P_\theta-z \psi)^{-1}W\\ \nonumber
&=(P_\theta-z \psi)^{-1}(P_\theta-iM Q-z \psi ),\\ \nonumber
&= I -(P_\theta-z \psi)^{-1}iMQ \\
&=: I + \widetilde{K}(z).
\label{eq:PWK3}
\end{align}
Thus
\begin{align}\label{eq:nursery0}
\log | \det (I + K(z))| = - \log |\det (I+ \widetilde{K}(z))|,
\end{align}
and we need an upper bound on $\log |\det (I+ \widetilde{K}(z_0))|$ for some $z_0\in \Omega$.

By \eqref{eq:B511} and \eqref{eq:boundQ}, 
\beq\label{eq:nursery1}
\log |\det (I+ \widetilde{K}(z))| \leq C h^{-d} \| (P_\theta-z\psi)^{-1}\|_{L^2\to L^2},
\eeq
and so we therefore need an upper bound on $\| (P_\theta-z\psi)^{-1}\|_{L^2\to L^2}$ for some $z_0\in \Omega$.

 \begin{lemma}[Cut-off resolvent bound for $\Im z>0$]\label{lem:psi}
Assume $\psi \in C_{\rm comp}^\infty(\Rea^d;\Rea)$.
Suppose that $\psi\geq c >0$ on a set that geometrically controls all
backward trapped rays for $P$. Then there exists $\smallpt>0$ such that the following is true. 
Given $\chi\in C^\infty_{\rm comp}(\mathbb{R}^d;\Rea)$, $h_0>0$, and a
choice of function $\Lambda(h)=o(h)$,
there exists $C>0$ such that, for all $z=z(h)$ with
$\abs{\Re z(h)}<\smallpt$, $0<\Im z(h)\leq \Lambda(h)$, and $0<h<h_0$,
\beq\label{eq:key}
\|\chi (P-z \psi-i0)^{-1}\chi \|_{L^2\to L^2} \leq \frac{C}{\Im z}.
\eeq
 \end{lemma}
 
 \begin{lemma}[Penetrable-obstacle cut-off resolvent bound for $\Im z>0$]\label{lem:psiT}
Given $n_i>0$ and $\obst\subset \mathbb{R}^d$ compact and Lipschitz, let $n$ be as in \eqref{eq:n}, and assume that $\supp \psi \supset \obst$ and there exists $c>0$ such that $\psi \geq c$ on $\obst$.
Given $\chi\in C^\infty_{\rm comp}(\mathbb{R}^d)$ and $k_0>0$ there exists $C>0$ such that,  for all $\Im z >0$ and $0<h<h_0$, 
\beq\label{eq:keyT}
\|\chi (
P -z \psi-i0)^{-1}\chi \|_{L^2\to L^2} \leq C h^{-2}\frac{ (1+|z|^2)}{\Im z}.
\eeq
\end{lemma}

 By \cite[Lemma 3.3]{GLS2}, Lemmas \ref{lem:psi} and \ref{lem:psiT} have the following corollary.

 \begin{corollary}[Bounds on the scaled operator for $\Im z>0$]
  \label{cor:psi}
 
 \
 
 (i) Under the assumptions of Lemma \ref{lem:psi}, 
given $c,h_0,\varepsilon>0$ there exists $C>0$
and $\smallpt>0$ such that, for all
$z=z(h)$ with
$\abs{\Re z(h)}<\smallpt$, $0<\Im z(h)=o(h)$,  $0<h<h_0$, and $\varepsilon \leq \theta \leq \pi/2-\varepsilon$, 
\beq\label{eq:key2}
\|(P_\theta-z \psi)^{-1} \|_{L^2\to L^2} \leq \frac{C}{\Im z}.
\eeq

(ii) Under the assumptions of Lemma \ref{lem:psiT}, given $h_0,\varepsilon>0$ there exists $C>0$ such that, for all $\Im z>0$, $0<h<h_0$,  and $\varepsilon \leq \theta \leq \pi/2-\varepsilon$, 
\beq\label{eq:key2T}
\|(P_\theta-z \psi)^{-1} \|_{L^2\to L^2} \leq C h^{-2}\frac{ (1+|z|^2)}{\Im z}.
\eeq
\end{corollary}

\begin{proof}
\cite[Lemma 3.3]{GLS2} shows that the scaled operator inherits the bound on the cut-off resolvent in the black-box setting, uniformly for the scaling angle $\varepsilon \leq \theta\leq \pi/2-\varepsilon$; the idea of the proof 
is to approximate the scaled operator away from the black-box using the free (i.e., without scatterer) scaled resolvent, and approximate it near the black-box using the unscaled resolvent (and then use, crucially, Lemma \ref{lem:agreement}).

\cite[Lemma 3.3]{GLS2} is written for $P-k^2$, where $P$ is a non-semiclassical black-box operator (in the sense of the second part of \cite[Definition 4.1]{DyZw:19}). 
Whereas $P = -h^2 \Delta - n$ can be written in that form (by dividing by $n$ and multiplying by $k^2$), $P-z\psi= -h^2 \Delta -n- z\psi$ cannot (because of the possibility of $n+ z\psi$ being zero). However, the proof of \cite[Lemma 3.3]{GLS2} goes through verbatim:~although  $P-z\psi$ is
not self-adjoint when $z$ is not real, and hence 
 not a semiclassical black-box operator (in the sense of the first part of \cite[Definition 4.1]{DyZw:19}), the only result from the black-box framework that is used in the proof of 
\cite[Lemma 3.3]{GLS2}, is agreement of the scaled and unscaled resolvents away from the scaling region, and this is established in our case 
by Lemma \ref{lem:agreement}.
\end{proof}

To prove Part (a) of Theorem \ref{thm:main}, we choose $z_0$ in \eqref{eq:Jensen}/\eqref{eq:nursery1} to be $ih^{1+\epsilon}$ for $\epsilon>0$, since this is, firstly, in $\Omega\supset \{ z: |z|< \smallp\}$ if $h< \smallp$ and, secondly, a $z_0$ for which Part (i) of Corollary \ref{cor:psi} applies (since $\Re z_0=0$).
By \eqref{eq:nursery0}, \eqref{eq:nursery1}, and \eqref{eq:key2}, under the assumptions of Theorem \ref{thm:main}, 
\begin{equation}\label{eq:Jensen2}
\log | \det (I + K(i h^{1+\epsilon}))| = - \log |\det (I+ \widetilde{K}(i h^{1+\epsilon}))| \geq - Ch^{-d-1-\epsilon}.
\eeq
Combining \eqref{eq:Jensen}, \eqref{eq:Jensen1}, and \eqref{eq:Jensen2}, 
we obtain
\beqs
\sum_{z\in \Omega} m_\theta(z) \leq C h^{-d-1-\epsilon};
\eeqs
Part (a) of Theorem \ref{thm:main} then follows from Corollary
\ref{cor:poles}.

The proof of Part (a) of Theorem \ref{thm:mainT} is very similar; the only difference is that, since \eqref{eq:key2T} is valid for all $\Im z>0$, instead of just for $0<\Im z=o(h)$ as in \eqref{eq:key2}, we now obtain a lower 
bound on $\log |\det (I+ \widetilde{K}(z_0))|$ by choosing a $z_0\in \Omega$ with constant imaginary part. Indeed, under the assumptions of Theorem \ref{thm:mainT}, by \eqref{eq:nursery0}, \eqref{eq:nursery1}, and \eqref{eq:key2T},
\begin{equation}\label{eq:Jensen2T}
\log | \det (I + K(i\smallp/2 ))| = - \log |\det (I+ \widetilde{K}(i\smallp/2))| \geq - Ch^{-d-2},
\eeq
and thus 
\beqs
\sum_{z\in \Omega} m_\theta(z) \leq C h^{-d-2}
\eeqs
and Part (a) of Theorem \ref{thm:mainT} follows.

 It therefore remains to prove Lemmas \ref{lem:psi} and \ref{lem:psiT}.

\begin{proof}[Proof of Lemma \ref{lem:psi}]
We first fix $\smallpt$. By the assumption that $\psi\geq c>0$ on a set that geometrically controls all
backward-trapped null bicharacteristics for $-k^{-2} \Delta -n$,
every point in $\Sigma_p$, with $p:=|\xi|^2 -n$ \eqref{eq:p}, reaches the
set 
\beq\label{eq:back1}
\{\psi\geq c\} \cup
\{\abs{x}>R_0,\ x\cdot \xi\leq 0\}
\eeq 
(i.e., either on the support of $\psi$ or
incoming) under the backward $H_p$ flow.
This dynamical hypothesis is stable under small perturbations of $p$.
In particular, if $\smallpt$ is sufficiently small then for all $\abs{z}\leq \smallpt$, by
compactness of the backward trapped set within a closed ball, it
remains true that $\{\psi\geq c\}$ geometrically controls all
backward-trapped null bicharacteristics of $p'= p- z \psi$ as well (note that this is the only place where smallness of $\smallpt$ plays a role).

Having fixed $\smallpt$, we now suppose the asserted bound \eqref{eq:key} fails.  Then there exist a sequence of functions
   $g_j$, along with sequences $h_j\to 0$ and $z(h_j)$
   with $$\delta(h_j) := \Im z(h_j) \leq \Lambda(h_j),\quad \Re z(h_j) \to z_0
   \in [-\smallpt,\smallpt]$$ such that
   $$
\norm{\chi (P-z(h_j) \psi-i0)^{-1}\chi g_j} \geq j \delta(h_j)^{-1}\norm{g_j}.
$$
Let $R_\chi>0$ be such that $\supp \chi \subset B(0,R_\chi)$.
Below, we will use the weakening of this inequality to
\begin{equation}\label{weaker}
\norm{(P-z(h_j) \psi-i 0)^{-1}\chi g_j}_{L^2(B(0,R_\chi)} \geq j
\delta(h_j)^{-1}\norm{\chi g_j}.
\end{equation}
Now set
$$
u_j=\frac{(P-z(h_j) \psi-i0)^{-1}(\chi g_j)}{\norm{(P-z(h_j) \psi-i0)^{-1}(\chi g_j)}_{L^2(B(0,R_\chi))}},
$$
and
$$
f_j=\frac{\chi g_j}{\norm{(P-z(h_j) \psi-i0)^{-1}(\chi g_j)}_{L^2(B(0,R_\chi))}}.
$$
Then by \eqref{weaker} and Lemma~\ref{lem:joey},
\begin{align}\label{eq:back2}
\norm{u_j}_{L^2(B(0,R_\chi))}&=1,\\  \nonumber
f_j&=(P-z(h_j) \psi)  u_j \text{ is supported in } \supp \chi,\\ \nonumber
\norm{f_j}&=o(\delta(h_j)).
\end{align}
Now pass to a subsequence so that
we may extract a defect measure $\mu$, i.e., a positive Radon measure
on $T^*\RR^d$ so that for any $A=\Op_h (a)$
supported in $T^*B(0, R_\chi)$,
$$
\ang{A u_j,u_j} \to \mu(a) := \int a \, d\mu
$$
see Theorem \ref{thm:defect_existence} below. 

By \eqref{factorization} and \cite[Propositions~2.2 and
  3.5]{Bu:02}, the sequence $u_j$ is \emph{outgoing} in the sense
that the measure $\mu$ vanishes
on a neighborhood of all \emph{incoming points}, i.e., those with $x\cdot \xi \leq
0$, $\abs{x}\geq R_0$.

Now return to the equation
\beq\label{eq:Sat1}
(P-z(h_j) \psi)  u_j=f_j,
\eeq
rearranged as
$$
(P-\Re z(h_j)\psi) u_j=f_j+i\Im z(h_j) \psi u_j.
$$
Since 
$\norm{f_j}=o(\delta(h_j))$ and $\delta(h)\leq \Lambda(h)=o(h)$,
the family
$u_j$ is an $o(\delta(h))=o(h)$-quasimode of the $h$-dependent family of
operators
$$
(P-\Re z(h_j)\psi),
$$
whose semiclassical principal symbols are converging to
$$p':=\abs{\xi}^2-n-z_0 \psi,\quad z_0 \in [-\smallpt,\smallpt].$$
By Theorem \ref{thm:characteristic} below, 
$\supp \mu\subset \Sigma_{p'}$, the characteristic set of $P-z_0 \psi$.
Since $\Sigma_{p'}$ has compact support in the fibers of
$T^*\RR^d$, we can make sense of $\mu(a)$ even when $a$ has noncompact
support in fiber directions---cf.\ \cite[Lemma 3.5]{GaSpWu:18}.

Multiplying  \eqref{eq:Sat1} by $u_j$ and  integrating by parts over $B(0, R_\chi)$ 
(recalling that $\supp f_j \subset B(0, R_\chi)$) yields
\begin{align*}
&\langle h \nabla u_j , h \nabla u_j \rangle_{B(0,R_\chi)} - \langle n u_j, u_j\rangle_{B(0,R_\chi)} - z \langle \psi u_j,u_j\rangle_{B(0,R_\chi)}
-h \langle {\rm DtN}u_j ,u_j\rangle_{\partial B(0,R_\chi)} \\
&\qquad= \langle f_j, u_j\rangle_{B(0,R_\chi)},
\end{align*}
where ${\rm DtN}$ is the Dirichlet-to-Neumann map ($u \mapsto h^{-1}\partial_\nu u$) for the constant-coefficient Helmholtz equation outside $B(0,R_\chi)$. 
Taking  the imaginary part of the last displayed equation
and recalling that $\Im \langle {\rm DtN}u ,u\rangle_{\partial B(0,R)}\geq 0$ (see, e.g., \cite[Equation 2.6.94]{Ne:01}), we find that
$$
\delta(h_j) \ang{\psi u_j, u_j}\leq \abs{\ang{f_j, u_j}} \leq
\norm{f_j} \norm{u_j}_{L^2(B(0, R_\chi))}=o(\delta(h_j));
$$
hence $\mu(\psi)=0.$  This implies in particular that $$\supp \mu \cap
T^*\{\psi\geq c\}=\emptyset.$$

We now turn to the propagation of defect measure.  
Let $\varphi_t$ denote the flow along $H_{p'}$; i.e., $\varphi_t(\cdot) := \exp({tH_{p'}}(\cdot))$.
By
Theorem \ref{thm:propagation} and Corollary \ref{cor:invariance} below 
  \footnote{See also the remark on 
  \cite[Page 388]{DyZw:19} to deal with the fact that the symbol of $P-\Re
  z(h_j) \psi$ is $h$-dependent.},
for all $t\in \RR$ and all Borel sets $U \subset \Sigma_{p'}$,
$\mu(U)=0$ implies $\mu(\varphi_t (U)))=0$.  In other words, the
support of the defect measure is invariant under the null
bicharacteristic flow. (Owing to our smallness assumptions on $\Im
z(h_j)$ this propagation holds both forward and backward in time, but we only
need the above propagation statement for $t\geq 0$.)

Recall from the start of the proof that, with $\smallpt$ is sufficiently small, if $\abs{z_0}\leq \smallpt$ 
then
every point in $\Sigma_{p'}$ reaches the
set \eqref{eq:back1} under the backward $H_{p'}$ flow.
 Thus, since $\mu$ vanishes on the set \eqref{eq:back1} that is reached by all backwards $H_{p'}$
null bicharacteristics, it vanishes identically; this
contradicts the assumption that 
$\mu(T^* B(0, R_\chi))=1$, coming from \eqref{eq:back2} (i.e., $u$ is $L^2$-normalized on
$B(0, R_\chi)$).
\end{proof}

We now turn to Lemma~\ref{lem:psiT}.
As described in \S\ref{sec:ideas}, the proof of Lemma \ref{lem:psiT} involves a commutator with $x\cdot \nabla$ (plus lower-order terms). This is conveniently written via the following integrated  identity. 

\begin{lemma}[Integrated form of a Morawetz identity]
\label{lem:morid1int}
Let $D$ be a bounded Lipschitz open set, with boundary $\partial D$ and outward-pointing unit normal vector $\nu $.
Given $\alpha,\beta \in \Rea$, let 
\beqs%\label{eq:cM}
\mathcal{M}_{\alpha,\beta} v:= \bx\cdot \gv - i h^{-1}\beta v + \alpha v.
\eeqs
If 
\beqs
v \in V(D):= \bigg\{ 
v\in H^1(D) :\;  \Delta v\in L^2(D), \;\partial_\nu v \in L^2(\partial D),\; v \in H^1(\partial D)
\bigg\}
\eeqs
and $n,\alpha,\beta\in\Rea$, then
\begin{align}\nonumber 
&\int_D 2 \Re \big\{\overline{\cM_{\alpha,\beta} v } \,(h^2 \Delta +n) v \big\}
+ (2\alpha -d +2) h^2 \ngvs  +(d-2\alpha) n  \nvs
\\
&=\int_{\partial D}(\bx\cdot\nu )\left(h^2  \left|\partial_\nu v\right|^2 -h^2 |\nabla_{\partial D} v|^2 + n \nvs\right)
+ 2h\Re\Big\{\big(  \bx\cdot\overline{\nabla_{\partial D} v}+ i h^{-1}\beta \vb + \alpha \vb\big) h\partial_\nu v\Big\},
\label{eq:morid1int}
\end{align}
where $\nabla_{\partial D}$ is the surface gradient on $\partial D$ (such that $\nabla v = \nabla_{\partial D} v + \nu \partial_\nu v$ for $v\in C^1(\overline{D})$.
\end{lemma}

We later use \eqref{eq:morid1int} with $\beta=R$, in which case all the terms in $\mathcal{M}_{\alpha,\beta} v$ are dimensionally homogeneous.

\begin{proof}[Proof of Lemma \ref{lem:morid1int}]
If $v\in C^\infty(\overline{D})$, then \eqref{eq:morid1int} follows from divergence theorem
applied to the identity 
\begin{align}\nonumber 
2 \Re \big\{\overline{\mathcal{M}_{\alpha,\beta} v } \,(h^2 \Delta +n) v \big\} = &\, \nabla \cdot \bigg[ 2 h\Re\big\{\overline{\cM_{\alpha,\beta} v}\,  h\gv\big\} + \bx\big(n |v|^2 -h^2 |\nabla v|^2\big)\bigg] \\
&\hspace{10ex}
- (2\alpha -d +2)h^2  |\nabla v|^2 -(d-2\alpha)n |v|^2
\label{eq:morid1}
\end{align}
which can be proved by expanding the divergence on the right-hand side.
By \cite[Lemmas 2 and 3]{CoDa:98}, $C^\infty(\overline{D})$ is dense in $V(D)$ and the result then follows 
since \eqref{eq:morid1int} is continuous in $v$ with respect to the topology of $V(D)$. 
\end{proof}

The following lemma is proved using the multiplier $\cM_{(d-1)/2, |x|} u$ (first introduced in \cite{MoLu:68}) and consequences of the Sommerfeld radiation condition; see, e.g., \cite[Proof of Lemma 4.4]{MoSp:19}. 

\begin{lemma}[Inequality on $\partial B(0,R)$ used to deal with the contribution from infinity] \label{lem:2.1}
Let $u$ be a solution of the homogeneous Helmholtz equation $(h^2 \Delta +1) u=0$ in $\Rea^d\setminus \overline{B_{R_0}}$ (with $d\geq 2$), for some $R_0>0$, satisfying the Sommerfeld radiation condition \eqref{eq:src}. 
Then, for $R>R_0$, 
\begin{align}\nonumber
&\int_{\partial B(0,R)} R\left(h^2  \left|\pdiff{u}{r}\right|^2 - h^2 |\nabla_{\partial B(0,R)} u|^2 + |u|^2\right)   - 2 R\, \Im \int_{\partial B(0,R)} \bar{u}h \pdiff{u}{r}
\\&\hspace{2cm} + (d-1)h\Re \int_{\partial B(0,R)}\bar{u}h\pdiff{u}{r} \leq 0,
\label{eq:2.1}
\end{align}
\end{lemma}

With Lemmas \ref{lem:morid1int} and \ref{lem:2.1} in hand, we now prove Lemma \ref{lem:psiT}.

\begin{proof}[Proof of Lemma \ref{lem:psiT}]
It is sufficient to prove that for any $R>0$ such that $\obst \subset B(0,R)$, given $f\in L^2_{\rm comp}(\mathbb{R}^d)$ with $\supp f \subset \overline{B(0,R)}$, 
the outgoing solution $u\in H^1_{\rm loc}(\mathbb{R}^d)$ to 
\beq\label{eq:PDET}
(-h^2 \Delta - n - z\psi)u = f \quad\text{ in }\mathbb{R}^d
\eeq
satisfies 
\beq\label{eq:STPT}
\| u\|_{L^2(B(0,R))}\leq Ch^{-2}\frac{ (1+|z|^2)}{\Im z} \| f\|_{L^2(B(0,R))}.
\eeq
Just as in the proof of Lemma \ref{lem:psi}, by multiplying \eqref{eq:PDET} by $\overline{u}$ and integrating over $B(0,R)$, 
\beqs
\int_{B(0,R)} h^{2} |\nabla u|^2 - n|u|^2 - z \psi|u|^2 - h\langle {\rm DtN}u ,u\rangle_{\partial B(0,R)} = \int_{B(0,R)} f \overline{u},
\eeqs
where ${\rm DtN}$ is the Dirichlet-to-Neumann map ($u \mapsto h^{-1}\partial_\nu u$) for the constant-coefficient Helmholtz equation outside $B_{R}$. 
As before, we
take the imaginary part of the last displayed equation
and use that  $\Im \langle {\rm DtN}u ,u\rangle_{\partial B(0,R)}\geq 0$ 
to obtain
\beq\label{eq:Green}
(\Im z)\int_{B(0,R)} \psi |u|^2 \leq -\Im \int_{B_{R_0}}f\overline{u};
\eeq 
our goal now is to control $\|u\|_{L^2(B(0,R))}$ in terms of $\|\psi^{1/2} u\|_{L^2(B(0,R))}$.

We now apply the identity \eqref{eq:morid1int} with $v=u$, $\alpha= (d-1)/2$, and $\beta= R$. We first choose $D = \obst$ and $n=n_i$ and then $D= B(0,R)\setminus \overline{\obst}$ and $n=1$. These applications of \eqref{eq:morid1int} is allowed, since the solution $u$ of \eqref{eq:PDET} when $\obst$ is Lipschitz is in $V(\obst)$ and $V(B(0,R)\setminus\overline{ \obst})$ by, e.g., \cite[Lemma 2.2]{MoSp:19}. Adding the two resulting identities, and then using Lemma \ref{lem:2.1} to deal with the terms on $\partial B(0,R)$, we obtain that
\begin{align}\nonumber
&\int_{\obst}h^2|\nabla u|^2 + n_i|u|^2  + \int_{B(0,R)\setminus \overline{\obst}} h^2|\nabla u|^2 + |u|^2 \\
&\leq 
- 2 \Re \int_{B(0,R)} \big(\overline{x \cdot \nabla u - i h^{-1} R u + (d-1)u/2}\big) ( f+ z\psi u) + \int_{\partial \obst} (x\cdot \nu) (n_i-1) |u|^2,
\label{eq:Morawetz}
\end{align}
where $\nu$ is the outward-pointing unit normal vector on $\partial
\obst$ (note that \eqref{eq:Morawetz} is contained in \cite[Equation 5.3]{MoSp:19}, where
the variables $\eta$, $a_o,$ $a_i$, $n_o$, $A_D$, $A_N$ in that equation are all set to one).
When $x\cdot \nu>0$ (i.e.,
$\obst$ is star-shaped) and $n_i<1$, the term in \eqref{eq:Morawetz}
on $\partial \obst$ has the ``correct" sign and 
then using the inequality 
\beq\label{eq:PeterPaul}
2ab \leq \epsilon a^2 + \epsilon^{-1}b^2\quad\text{ for } \quad a,b,\epsilon>0.
\eeq
in \eqref{eq:Morawetz} gives the bound
$h \| \nabla u \|_{L^2(B(0,R))} + \|u\|_{L^2(B(0,R))} \leq C k
\|f\|_{L^2}$ when $z=0$. Since we also want to consider $n_i>1$, and
we have control of $\|u\|^2_{L^2(\obst)}$ via \eqref{eq:Green}, we
instead recall the multiplicative trace inequality
   (see, e.g.,
\cite[Theorem 1.5.1.10]{Gr:85})
\begin{align}
\| u\|^2_{L^2(\partial \obst)} 
&\leq C\Big( \epsilon h^2 \|\nabla u\|^2_{L^2(\obst)} + h^{-2} \epsilon^{-1} \|u\|^2_{L^2(\obst)}\Big),
\label{eq:mult_trace}
\end{align}
for $\epsilon h^2<1$.

By \eqref{eq:Morawetz}, given $h_0>0$ there exists $C>0$ such that, for $0<h<h_0$, 
\begin{align}\nonumber
&h^{2} \|\nabla u\|^2_{L^2(B(0,R))} + \|u\|^2_{L^2(B(0,R))} \\ \nonumber
&\leq C h^{-1} \Big( h\|\nabla u\|_{L^2(B(0,R))}+\|u\|_{L^2(B(0,R))}\Big)\Big( \|f\|_{L^2(B(0,R))}+ |z| \|\psi^{1/2} u\|_{L^2(B(0,R))}\Big) \\
&\qquad+ \|u\|^2_{L^2(\partial\obst)}.\label{eq:boxing1}
\end{align}
Using in \eqref{eq:mult_trace} that $\psi \geq c$ on $\obst$, we obtain that 
\begin{align}
\| u\|^2_{L^2(\partial \obst)} 
&\leq C\Big( \epsilon h^2 \|\nabla u\|^2_{L^2(B(0,R))} + h^{-2} \epsilon^{-1} \|\psi^{1/2} u\|^2_{L^2(B(0,R))}\Big),
\label{eq:mult_trace2}
\end{align}
for $\epsilon h^2<1$. 
Then using \eqref{eq:mult_trace2} in the last term on the right-hand side of \eqref{eq:boxing1}, and \eqref{eq:PeterPaul} 
on the other terms, we find that 
\begin{align*}
&h^{2} \|\nabla u\|^2_{L^2(B(0,R))} + \|u\|^2_{L^2(B(0,R))} \\
&\leq C \Big ( h^{-2} \|f\|^2_{L^2(B(0,R))} + \epsilon h^2\|\nabla u\|^2_{L^2(\obst)} + ( 1 + h^{-2} \epsilon^{-1}) (1 + |z|^2) \|\psi^{1/2}u\|^2_{L^2(B(0,R))} \Big),
\end{align*}
By choosing $\epsilon$ sufficiently small, and then using \eqref{eq:Green}, we find that 
\begin{align*}
&h^{2} \|\nabla u\|^2_{L^2(B(0,R))} + \|u\|^2_{L^2(B(0,R))} \\
&\leq C \left( h^{-2} \|f\|^2_{L^2(B(0,R))}  + h^{-2} \frac{(1 + |z|^2)}{\Im z} \|f \|_{L^2(B(0,R))} \|u\|_{L^2(B(0,R))} \right),
\end{align*}
and the required result \eqref{eq:STPT} then follows from one last application of \eqref{eq:PeterPaul}.
\end{proof}

\section{Bounds on the solution-operator for real $z$ (proofs of Part (b) of Theorems \ref{thm:main} and \ref{thm:mainT}).}\label{sec:resolvent}

\begin{theorem}[Variant of semiclassical maximum principle \cite{TaZw:98, TaZw:00}]
\label{thm:scmp}
Let $\cH$ be an Hilbert space and $z\mapsto Q(z,h)\in\mathcal{L}(\cH)$
an holomorphic family of operators in a neighbourhood of
\beqs
\Omega(h):=\big(w-2a(h),w+2a(h)\big)+i\big(-\delta(h)h^{-L},\delta(h)\big),
\eeqs
where 
\beq\label{eq:restrict1}
0<\delta(h)<1,\qquad \tand \quad a(h)^{2}\geq Ch^{-3L}\delta(h)^{2}
\eeq
for some $L,C>0$. 
Suppose that
\begin{align}\label{eq:expbound}
\Vert Q(z,h)\Vert_{\cH\rightarrow\cH}&\leq\exp(Ch^{-L}),\qquad z\in\Omega,\\ \label{eq:absbound}
\Vert Q(z,h)\Vert_{\cH\rightarrow\cH}&\leq \frac{b(h)}{\Im
                                       z},\qquad \text{ for }
                                       \Im z >0,\ z \in \Omega,
\end{align}
with $b(h)\geq 1$.
Then, 
\beq\label{eq:scmp1}
\Vert Q(z,h)\Vert_{\cH\rightarrow\cH}\leq b(h) \delta(h)^{-1}\exp(C+1),\quad \tfa z\in\big[w-a(h),w+a(h)\big].
\eeq
\end{theorem}

\begin{proof}[References for proof]
Let $f,g\in \cH$ with $\|f\|_{\cH}=\|g\|_{\cH}=1$, and let
\[
F(z,h):=\big\langle Q(z+w,h)g,f\big\rangle_{\cH}.
\]
The result \eqref{eq:scmp1} follows from the ``three-line theorem in a rectangle'' (a consequence of the maximum principle) stated as \cite[Lemma D.1]{DyZw:19} applied to the holomorphic family $(F(\cdot,h))_{0<h\ll1}$
with
\begin{align*}
R=2a(h),\qquad \delta_{+}=\delta(h),\qquad \delta_{-}=\delta(h)h^{-L},\\
M=M_{-}=\exp(Ch^{-L}),\qquad M_{+}=b(h)/\delta(h).
\end{align*}
Indeed, the condition \cite[Equation D.1.3]{DyZw:19}
\beqs
R^2 \delta_-^{-2} \geq \log\left(\frac{M}{\min_{\pm} M_{\pm}}\right)
\eeqs
becomes 
\beqs
a(h)^2 \delta(h)^{-2} h^{2L}\geq \log \left( \frac{\exp (Ch^{-L})}{b(h)/\delta(h)}\right) = C h^{-L} + \log (\delta(h)/b(h)),
\eeqs
which is ensured by \eqref{eq:restrict1} since $\delta(h)/b(h)<1$ and thus $ \log (\delta(h)/b(h))<0$.
\end{proof}

Part (b) of Theorems \ref{thm:main} and \ref{thm:mainT} is proved below using Theorem \ref{thm:scmp},
with $Q(z,h)= \chi (P - z\psi-i0)^{-1}\chi$, $\mathcal{H}=L^2$, 
 \eqref{eq:key}/\eqref{eq:keyT} providing the bound \eqref{eq:absbound}, and the following lemma providing the bound \eqref{eq:expbound}.

\begin{lemma}[Bounds on $(P -z\psi-i0)^{-1}$ away from poles]\label{thm:TZkey}
Given $\epsilon>0$, if the hypotheses of Theorem \ref{thm:main} hold, let $M:= d+1+\epsilon$. 
If the hypotheses of Theorem \ref{thm:mainT} hold, let $M:= d+2$. 
Let $\Omega\Subset \mathbb{C}$ containing the origin.
Let $h \mapsto g(h)$ be a positive function strictly bounded from above by $1$. Then 
there exist $h_0>0$ and $C_1>0$ such that, for $0<h<h_0$,
\begin{align}\label{eq:TZkey1}
\|\chi (P - z\psi-i0)^{-1}\chi \|_{L^2 \to L^2}
\leq C_1\exp\left(C_1 h^{-M} \log \left(\frac{1}{g(h)}\right) \right)\\
\qquad \text{ for all }\,
z \in \Omega \setminus \bigcup_{z_j \in \mathcal{P}} D(z_j, g(h)),\nonumber
\end{align}
where $\mathcal{P}$ is the set of poles of $(P -z\psi-i0)^{-1}$ and $D(z_j, g(h))$ is the open disc of radius $g(h)$ centred at $z_j\in \mathbb{C}$.
\end{lemma}

\begin{proof}
We follow the proof of \cite[Theorem 7.5]{DyZw:19}, noting that many steps are similar to those in \S\ref{sec:poly_bound}.
First, by Lemma \ref{lem:agreement}, $\chi (P-z\psi)^{-1}\chi = \chi(P_\theta-z\psi)^{-1}\chi$, from which 
\begin{equation}
\label{eq:thetaL2}
\|\chi (P - z\psi)^{-1}\chi \|_{L^2 \to L^2} \le \|(P_\theta-z\psi)^{-1}\|_{L^2\to L^2}.
\end{equation}
Thus it suffices to estimate the right-hand side. By \eqref{eq:PWK}, $P_\theta-z\psi = W(I+K(z))$, where 
\[W = P_\theta - z\psi - iMQ,\quad K(z) = W^{-1}(iMQ),\quad Q = \chi(hD)\chi(x)^2\chi(hD),\]
with $W$ uniformly invertible (by Lemma \ref{lem:W}) and $K(z)$ compact. Thus
\begin{equation}
\label{eq:thetatoK}
\|(P_\theta-z\psi)^{-1}\|_{L^2\to L^2} = \|(I+K(z))^{-1}W^{-1}\|_{L^2\to L^2}\le C\|(I+K(z))^{-1}\|_{L^2\to L^2}.
\end{equation}
Because $K(z)$ is trace class, by Part (iii) of Theorem \ref{thm:trace}, 
\begin{equation}
\label{eq:b521}
\|(I+K(z))^{-1}\|_{L^2\to L^2}\le |\det(I+K(z))|^{-1}\det(I+[K(z)^*K(z)]^{1/2}).
\end{equation}
Then, by Part (ii) of Theorem \ref{thm:trace} and \eqref{eq:square_root},
\[\log|\det(I+[K(z)^*K(z)]^{1/2})| \le \|[K(z)^*K(z)]^{1/2}\|_{\mathcal{L}_1} = \|K(z)\|_{\mathcal{L}_1}.\]
By \eqref{eq:boundK} and \eqref{eq:boundQ}, $\|K(z)\|_{\mathcal{L}_1}\le Ch^{-d}$; thus 
\begin{equation}
\label{eq:detK*K}
|\det(I+[K(z)^*K(z)]^{1/2})| \le \exp\big(Ch^{-d}\big).
\end{equation}
On the other hand, a consequence of Jensen's formula is that for any function $f$ holomorphic on a neighborhood of $\Omega$ and any $z_0\in\Omega$, that there exists $C$ such that
\beq\label{eq:missingterm}
\log|f(z)|-\log|f(z_0)| \ge -C\log\left(\frac{1}{\delta}\right)\left(\max_{z\in\Omega} \log |f(z)|  - \log |f(z_0)|\right)
\eeq
for all $z$ away from the zeros of $f$; see \cite[Equation D.1.13]{DyZw:19}.
\footnote{Note that \cite[Equation D.1.13]{DyZw:19} does not contain the $-\log|f(z_0)|$ on the left-hand side. To see why this term is necessary, observe that without it the right-hand side of \eqref{eq:missingterm} is invariant under multiplication of $f$ by a non-zero scalar, whereas the left-hand side is not.} (In principle, $C$ in \eqref{eq:missingterm} depends on $z_0$, but since $\Omega$ is compact one can choose $C$ depending only on $\Omega$.)
Applying this to $f(z) = \det(I+K(z))$, $\delta = g(h)$, and
either $z_0 = ih^{1+\epsilon}$ if the hypotheses of Theorem \ref{thm:main} hold
or $z_0 = i\smallp/2$ if the hypotheses of Theorem \ref{thm:mainT} hold, and recalling the bounds
\[\log|\det(I+K(z))|\le Ch^{-d},\quad \log|\det(I+K(z_0))|\ge -Ch^{-M}\]
from \eqref{eq:Jensen1} and \eqref{eq:Jensen2}/\eqref{eq:Jensen2T}, we obtain
\begin{align*}
\log|\det(I+K(z))|&\ge -C\log\left(\frac{1}{g(h)}\right)\left(Ch^{-d}+Ch^{-M}\right)-Ch^{-M}\\
&\ge -C'h^{-M}\log\left(\frac{1}{g(h)}\right),
\end{align*}
i.e.\
\begin{equation}
\label{eq:detK}
|\det(I+K(z))|^{-1} = \exp(-\log|\det(I+K(z))|) \le \exp\left(C'h^{-M}\log\left(\frac{1}{g(h)}\right)\right);
\end{equation}
the result follows by combining \eqref{eq:thetaL2}, \eqref{eq:thetatoK}, \eqref{eq:b521}, \eqref{eq:detK*K}, and \eqref{eq:detK}.
\end{proof}

We now prove Part (b) of Theorems \ref{thm:main} and \ref{thm:mainT}. This proof is similar to the proof of \cite[Theorem 3.3]{LSW1} (the proof that the resolvent is polynomially bounded for ``most'' $k\in [k_0,\infty)$), but is simpler because here we work with $z$ in a bounded interval, whereas \cite[Theorem 3.3]{LSW1} works with $k$ in the unbounded interval $[k_0,\infty)$. We give the proof for Part (b) of Theorem \ref{thm:main}, and outline the (small) changes needed for Part (b) of Theorem \ref{thm:mainT} at the end.

We will apply the semiclassical maximum principle (Theorem \ref{thm:scmp}) to sufficiently many rectangles of the form 
\beq\label{eq:rectangle}
\big(w-2a(h),w+2a(h)\big)+i\big(-\delta(h)h^{-L},\delta(h)\big).
\eeq
By \eqref{eq:restrict1}, we need that 
\beq\label{eq:delta_cond}
a(h)^{2}\geq Ch^{-3L}\delta(h)^{2}, \quad\text{ i.e., }\quad \delta(h)\le C'h^{3L/2}a(h),
\eeq
and this implies that  $\delta(h)\ll \delta(h)h^{-L}\ll a(h)$ as $h\to 0$.

With $\smallpt$ given by Lemma \ref{lem:psi}, we choose $\smallp$ to be slightly smaller, say, $\smallp:= 9\smallpt/10$.
The reason for this is that we will apply the semiclassical maximum principle to rectangles of the form \eqref{eq:rectangle} (with $a(h)\to 0$) for, in principle, arbitrary $w\in (-\smallp,\smallp)$, and we need to ensure that 
$(w-2a(h),w+2a(h))\subset (-\smallpt,\smallpt)$ so that the resolvent estimate of Lemma \ref{lem:psi} holds for all $z\in (w-2a(h),w+2a(h))$.

Let $\mathcal{P}$ denote the set of poles in $\{ z:
\abs{z}<\smallpt\}$, and let $N(h)$ be their number. 
From Part (a) of Theorem \ref{thm:main}, 
we know that
\beq\label{eq:plane1}
N(h)\le Ch^{-M}
\eeq
where $M:= d+1+ \epsilon$. Let
\[E = \bigcup_{p\in\mathcal{P}} B(p,4a(h)) = \{w\,:\,\text{dist}(w,\mathcal{P})<4a(h)\}.\]
Given $w\in (-\smallp,\smallp)\setminus  E$, let
\beq\label{eq:Omega_w}
\Omega_w = (w-2a(h),w+2a(h)) + i(-\delta(h)h^{-L},\delta(h)),
\eeq
and observe that, for all $z\in \Omega_w$, $\text{dist}(z,\mathcal{P})>a(h)$ (since $\text{dist}(z,w)\le 2a(h)+\delta(h)h^{-L} < 3a(h)$
 for $h\ll 1$, and $\text{dist}(w,\mathcal{P})\ge 4a(h)$). Therefore,
 for $w\in (-\smallp,\smallp)\setminus  E$, the result \eqref{eq:scmp1} of the semiclassical maximum principle gives a good
resolvent bound on the interval $[w-a(h),w+a(h)]$; in particular, a good resolvent bound at $w$
(see \eqref{Xmaseve:1} below).  
Before stating this resolvent bound, we need to restrict $a(h)$ so that the measure of the set $(-\smallp,\smallp)\cap E$ is $\leq \widetilde{\delta} h^N$ (to prevent a notational clash with $\delta(h)$ used in the semiclassical maximum principle, we relabel $\delta$ in Theorems \ref{thm:main} and \ref{thm:mainT} as $\widetilde{\delta}$ here).

By the definition of $E$,
$$
(-\smallp,\smallp)\cap E = \bigcup_{p\in\mathcal{P}} (-\smallp,\smallp)\cap B(p,4a(h)) =: S_k,
$$
 and $|(-\smallp,\smallp)\cap B(p,4a(h))|\le 8a(h)$ regardless of $p$. Therefore
\[
|S_k|=
|(-\smallp,\smallp)\cap E| \le 8a(h)N(h)
\]
and so, using part (a) of Theorems~\ref{thm:main} and \ref{thm:mainT},
$|S_k| \leq \widetilde{\delta} h^N$
will be ensured by
\[a(h)\le C'\widetilde{\delta} h^{M+N}.\]
We therefore now choose $$a(h) = C'\widetilde{\delta} h^{M+N}.$$
The condition \eqref{eq:delta_cond} on $\delta(h)$ then reduces to
\beq\label{eq:delta_cond2}
\delta(h)
\le C''\widetilde\delta h^{3L/2+M +N}.
\eeq
Having now established how big $a(h)$ and $\delta(h)$ can be in our application of Theorem \ref{thm:scmp}, we now determine the constant $L$ in \eqref{eq:expbound}.
Since $\text{dist}(z,\mathcal{P})>a(h)$ for  $z\in \Omega_w$, Lemma \ref{thm:TZkey} implies that,
on the bottom edge of the rectangle $\Omega_w$ \eqref{eq:Omega_w},
\begin{align*}
\|Q(z,h)\|&\le C_1\exp\left(C_2 h^{-M}\log\left(\frac{1}{a(h)}\right)\right),\\
&\le C_1\exp\left(C_3 h^{-M}\left((M+N)\log(h^{-1})\right)\right).
\end{align*}
Thus, given $\epsilon'>0$, there exists $C_4>0$ such that 
\[\|Q(z,h)\|\le C_1\exp\big(C_4 h^{-(M+\epsilon')}\big),\]
and we may therefore choose $L = M+\epsilon'$. Therefore, by \eqref{eq:delta_cond2}, we can set 
\[\delta(h) =  C''\widetilde\delta h^{5M/2+ N +3\epsilon'/2}.\]
Under the assumptions of Theorem \ref{thm:main}, 
on the upper edge of the rectangle $\Omega_w$ \eqref{eq:Omega_w},
$\|Q(z,h)\|\le C\delta(h)^{-1}$ by \eqref{eq:key}. Therefore, \eqref{eq:scmp1} implies that, 
for $z\in (-\smallp,\smallp)\setminus S_k$, 
where $|S_k|\leq \widetilde{\delta} h^N$, 
\begin{equation}\label{Xmaseve:1}\|Q(z,h)\|\le C\delta(h)^{-1} \le  C \widetilde\delta^{-1} h^{-5M/2 - N-3 \epsilon'/2},\end{equation}
which is \eqref{eq:resolvent}, recalling that in this case $M=d+1+\epsilon$ and absorbing $3 \epsilon'/2$ into $\epsilon$ (since both $\epsilon$ and $\epsilon'$ were arbitrary).

The changes to the above proof for Part (b) of Theorem \ref{thm:mainT} are the following.
\begin{itemize}
\item Since there is no restriction on the real parts of $z$ in Lemma \ref{lem:psiT}, given $\smallp>0$ we choose $\smallpt> \smallp$ (say, $\smallpt:= 10\smallp/9$). 
\item Now $M:= d+3$ in \eqref{eq:plane1}. 
\item When applying the semiclassical maximum principle, on the upper edge of the rectangle $\Omega_w$ there is an additional $h^{-2}$ (compare \eqref{eq:key} to \eqref{eq:keyT} and recall that $\delta(h)\ll 1$ so the $1+ |z|^2$ on the right-hand side of \eqref{eq:keyT} is effectively $1$).
This additional factor of $h^{-2}$, along with the new definition of $M$, leads to 
$5(d+1)/2$ in the exponent of the bound \eqref{eq:resolvent} changing to $2 + 5(d+3)/2$ in \eqref{eq:resolventT}.
\end{itemize}

\begin{appendix}

\section{Recap of relevant results from semiclassical analysis}\label{app:A}

\subsection{Weighted Sobolev spaces}\label{sec:SC1}

The \emph{semiclassical Fourier transform} is defined by 
$$
(\mathcal F_{h}u)(\xi) := \int_{\mathbb R^d} \exp\big( -i x \cdot \xi/h\big)
u(x) \, d x,
$$
with inverse
\beqs%\label{eq:SCFTinverse}
(\mathcal F^{-1}_{h}u)(x) := (2\pi h)^{-d} \int_{\mathbb R^d} \exp\big( i x \cdot \xi/h\big)
u(\xi)\, d \xi;
\eeqs
see \cite[\S3.3]{Zw:12}; i.e., the semiclassical Fourier transform is just the usual Fourier transform with the transform variable scaled by $h$. 
These definitions imply that, with $D:= -i \partial$,
\beq\label{eq:FTelement}
\cF_h \big( (h D)^\alpha) u\big) = \xi^\alpha \cF_h u \quad \tand\quad \norm{u}_{L^2(\Rea^d)} = \frac{1}{(2\pi h)^{d/2}}\norm{\cF_h u}_{L^2(\Rea^d)}; 
\eeq
see, e.g., \cite[Theorem 3.8]{Zw:12}.
Let 
\beq\label{eq:Hsk}
H_h^s(\Rea^d):= \Big\{ u\in \mathcal{S}'(\Rea^d) \,\text{ such that }\, \langle \xi\rangle^s (\cF_h u) \in L^2(\Rea^d) \Big\},
\eeq
where $\langle \xi \rangle := (1+|\xi|^2)^{1/2}$, $\mathcal{S}(\Rea^d)$ is the Schwartz space (see, e.g., \cite[Page 72]{Mc:00}), and $\mathcal{S}'(\Rea^d)$ its dual.
Define the norm
\beq\label{eq:Hhnorm}
\Vert u \Vert_{H_h^m(\Rea^d)} ^2 = \frac{1}{(2\pi h)^{d}} \int_{\Rea^d} \langle \xi \rangle^{2m}
 |\mathcal F_h u(\xi)|^2 \, d \xi.
\eeq
The properties \eqref{eq:FTelement} imply that the space $H_h^s$ is the standard Sobolev space $H^s$ with each derivative in the norm weighted by $h$.

\subsection{Semiclassical pseudodifferential operators}

A symbol is a function on $T^*\Rea^d$ that is also allowed to depend on $h$, and can thus be considered as an $h$-dependent family of functions.
Such a family $a=(a_h)_{0<h\leq h_0}$, with $a_h \in C^\infty({T^*\mathbb R^d})$, 
is a \emph{symbol
of order $m$}, written as $a\in S^m(T^*\Rea^d)$,
if for any multiindices $\alpha, \beta$
\beqs%\label{eq:Sm}
| \partial_x^\alpha \partial^\beta_\xi a_h(x,\xi) | \leq C_{\alpha, \beta} 
\langle \xi\rangle^{m -|\beta|}
\quad\tfa (x,\xi) \in T^* \Rea^d \text{ and for all } 0<h\leq h_0,
\eeqs
and 
$C_{\alpha, \beta}$ does not depend on $h$; see \cite[p.~207]{Zw:12}, \cite[\S E.1.2]{DyZw:19}. 

We now fix $\chi_0\in C_c^\infty(\mathbb{R})$ to be identically 1 near 0.  
We then say that an operator $A:C_{\rm comp}^\infty(\mathbb{R}^d)\to \mathcal{D}'(\mathbb{R}^d)$ is a semiclassical pseudodifferential operator of order $m$, and write $A\in \Psi^m_\hsc(\mathbb{R}^d)$, if $A$ can be written as
\begin{equation}
\label{e:basicPseudo}
Au(x)=\frac{1}{(2\pi h)^d}\int_{\Rea^d}\int_{\Rea^d} e^{\frac{i}{h}\langle x-y,\xi\rangle}a(x,\xi)\chi_0(|x-y|)u(y)dyd\xi +E
\end{equation}
where $a\in S^m(T^*\mathbb{R}^d)$ and $E=O(h^\infty)_{\Psi^{-\infty}}$, where an operator $E=O(h^\infty)_{\Psi^{-\infty}}$ if for all $N>0$ there exists $C_N>0$ such that
$$
\|E\|_{H_h^{-N}(\mathbb{R}^d)\to H_h^N(\mathbb{R}^d)}\leq C_Nh^N. 
$$
 We use the notation $a(x,hD_x)$ for the operator $A$ in~\eqref{e:basicPseudo} with $E=0$. The integral in \eqref{e:basicPseudo} need not converge, and can be understood \emph{either} as an oscillatory integral in the sense of \cite[\S3.6]{Zw:12}, \cite[\S7.8]{Ho:83}, \emph{or} as an iterated integral, with the $y$ integration performed first; see \cite[Page 543]{DyZw:19}.

We use the notation $a \in h^l S^m$  if $h^{-l} a \in S^m$; similarly 
$A \in h^l \Psi_h^m$ if 
$h^{-l}A \in \Psi_h^m$.

\subsection{The principal symbol map $\sigma_{\hsc}$}
Let the quotient space $ S^m/\hsc S^{m-1}$ be defined by identifying elements 
of  $S^m$ that differ only by an element of $\hsc S^{m-1}$. 
For any $m$, there is a linear, surjective map
$$
\sigma^m_{\hsc}:\Psi_\hsc ^m \to S^m/\hsc S^{m-1},
$$
called the \emph{principal symbol map}, 
such that, for $a\in S^m$,
\beq\label{eq:symbolone}
\sigma_\hsc^m\big(\Op_\hsc(a)\big) = a \quad\text{ mod } \hsc S^{m-1};
\eeq
see \cite[Page 213]{Zw:12}, \cite[Proposition E.14]{DyZw:19} (observe that \eqref{eq:symbolone} implies that 
$\operatorname{ker}(\sigma^m_{\hsc}) = \hsc\Psi_\hsc ^{m-1}$).
When applying the map $\sigma^m_{\hsc}$ to 
elements of $\Psi^m_\hsc$, we denote it by $\sigma_{\hsc}$ (i.e.~we omit the $m$ dependence) and we use $\sigma_{\hsc}(A)$ to denote one of the representatives
in $S^m$ (with the results we use then independent of the choice of representative).

\begin{lemma}(\cite[Propositions E.19 and E.24]{DyZw:19}~\cite[Theorem 13.13 ]{Zw:12})\label{lem:norm_symbol}
If $A=a(x,hD)\in \Psi^0_\hsc$, then there exists $C>0$ such that 
$$
\|A\|_{L^2\to L^2}\leq \sup_{(x,\xi)\in T^*\Rea^d} |a(x,\xi)|+ Ch.
$$
\end{lemma}

\subsection{Ellipticity}

To deal with the behavior of
functions on phase space uniformly near $\xi=\infty$ (so-called \emph{fiber infinity}), we consider the \emph{radial
  compactification} in the $\xi$ variable of $T^*\Rea^d$. This is defined by
$$
\overline{T} ^* \mathbb R^d:= \mathbb R^d \times B^d,
$$
where $B^d$ denotes the closed unit ball, considered as the closure of the
image of $\mathbb R^d$ under the radial compactification map 
$$\RC: \xi \mapsto \xi/(1+\langle \xi
\rangle);$$
see \cite[\S E.1.3]{DyZw:19}.
Near the boundary of the
ball, $\lvert \xi\rvert^{-1}\circ \RC^{-1}$ is a smooth function, vanishing to
first order at the boundary, with $(\lvert \xi\rvert^{-1}\circ \RC^{-1}, \widehat\xi\circ\RC^{-1})$
thus giving local coordinates on the ball near its boundary.  The boundary of the
ball should be considered as a sphere at infinity consisting of all
possible \emph{directions} of the momentum variable.  

We now give a simplified version of semiclassical elliptic regularity;
for the proof of this, as well as a statement and proof of the more-general version, see, e.g., \cite[Theorem E.33]{DyZw:19}.
For this, we say that $B\in \Psi^{m}_{\hsc}$ is \emph{elliptic} on $\overline{T^*\mathbb{R}^d}$ if 
\beqs%\label{eq:elliptic}
\liminf_{\hsc\to 0}\inf_{(x,\xi)\in  \overline{T^*\mathbb{R}^d}}\big|\sigma_\hsc(B)(x,\xi)\langle \xi\rangle^{-m}\big|>0.
\eeqs
\begin{theorem}[Simplified semiclassical elliptic regularity]\label{thm:elliptic}
If $B\in \Psi^{m}_{\hsc}(\Rea^d)$ is elliptic on $\overline{T^*\mathbb{R}^d}$ 
then there exists $\hsc_0>0$ such that, for all $0<\hsc\leq \hsc_0$, $B^{-1} : H_\hsc^{s-m}(\Rea^d) \rightarrow H_\hsc^{s}(\Rea^d)$ exists and is bounded (with norm independent of $\hsc$) for all $s$.
\end{theorem}

\subsection{Defect measures}

We say that $a\in S^{\rm comp}$ if $a\in S^{-\infty}$ and $a$ is compactly supported, and we say that $A\in \Psi^{\rm comp}_\hsc$ if $A\in \Psi_\hsc^{-\infty}$ and can be written in the form~\eqref{e:basicPseudo} with $a\in S^{\rm comp}$. 

\begin{definition}[Defect measure]
Given $\{u(\hsc)\}_{0<\hsc\leq\hsc_0}$, uniformly locally bounded, and a sequence $\hsc_n\to 0$, $\{u(\hsc)\}_{0<\hsc\leq\hsc_0}$ has defect measure $\mu$ if, for all $a\in S^{\comp}$,
\beq\label{eq:defect}
\lim_{n\to \infty}\big\langle \Op(a) u(\hsc_n),u(\hsc_n) \big\rangle= \int_{T^*\Rea^d} a\, \rd \mu.
\eeq
\end{definition}
Observe that \eqref{eq:defect} implies that if $A$ is the quantisation of a symbol $a\in S^{\comp}$, then 
\beqs
\lim_{n\to \infty}\big\langle A u(\hsc_n),u(\hsc_n) \big\rangle= \int_{T^*\Rea^d} \sigma_\hsc(A)\, \rd \mu;
\eeqs
indeed, this follows since $A- \Op(\sigma_\hsc(A)) \in \hsc \Psi^{-\infty}_\hsc$, by the definition of the principal symbol and the fact that $A\in \Psi^{-\infty}_\hsc$.

\begin{theorem}[Existence of defect measures]
\label{thm:defect_existence}
Suppose that $\{u(\hsc)\}_{0<\hsc\leq\hsc_0}$ is uniformly locally bounded and $\hsc_n\to 0$.
Then there exists a subsequence $\{\hsc_{n_\ell}\}_{\ell=1}^\infty$ and a Radon measure $\mu$ on $T^*\Rea^d$ such that 
 $\{u(\hsc_{n_\ell})\}_{\ell=1}^\infty$ has defect measure $\mu$.
\end{theorem}

\begin{proof}[References for the proof]
See, e.g., \cite[Theorem 5.2]{Zw:12}, \cite[Theorem E.42]{DyZw:19}.
\end{proof}

\begin{theorem}[Support of defect measure]
\label{thm:characteristic} 
Let $P\in \Psi^{m}_\hsc$ be properly supported. Suppose that $\{u(\hsc_n)\}$ has defect measure $\mu$, and satisfies
$$
\N{P u(\hsc_n)}_{L^2(\Rea^d)}\to 0 \quad\text{ as } n\to\infty.
$$
Then $\mu(\{\sigma_\hsc(P)\neq 0\})=0$; i.e., if $\supp\, a\subset \{\sigma_\hsc(P)\neq 0\}$, then $\int a\, \rd \mu=0$.
\end{theorem}

\begin{proof}[References for the proof]
See \cite[Equation 3.17]{Bu:02}, \cite[Theorem 5.3]{Zw:12}, \cite[Theorem E.43]{DyZw:19}.
\end{proof}

\begin{theorem}[Propagation of the defect measure under the flow]
\label{thm:propagation}
Suppose that  $P\in \Psi^{m}_\hsc$ is properly supported and formally self adjoint; denote its (real valued) principal symbol by $p$. 
Suppose that $\{u(\hsc_n)\}$ has defect measure $\mu$, and 
\beqs%\label{eq:Puoh}
\N{Pu(\hsc_n)}_{L^2(\Rea^d)}=o(h_n) \quad\text{ as } n\to \infty.
\eeqs
Then
\beqs%\label{eq:invariance}
\int 
H_p a 
 \,\rd \mu=0 \quad\tfa a \in S^{\rm comp}.
\eeqs
\end{theorem}

\begin{proof}[References for the proof]
See  \cite[Theorem 5.2]{Zw:12}, \cite[Theorem  E.44]{DyZw:19}.
\end{proof}

Let $\varphi_t$ denote the flow along $H_p$; i.e., $\varphi_t(\cdot) := \exp({tH_{p}}\cdot)$.

\begin{corollary}[Invariance under the flow written in terms of sets]
\label{cor:invariance}
Under the assumptions of Theorem \ref{thm:propagation}, given a Borel set $B\subset T^*\Rea^d$, 
\beqs%\label{eq:invariance_set}
\mu\big(\varphi_t(B)\big) = \mu(B) \quad\tfa t,
\qquad\text{ i.e., }\quad
\int 1_{\varphi_t(B)} \,\rd \mu = \int 1_B\, \rd \mu \quad\tfa t.
\eeqs
\end{corollary}

\section{Recap of relevant results about Fredholm and trace-class operators }\label{app:B}

\begin{theorem}[Analytic Fredholm theory]\label{thm:analytic_Fredholm}
Suppose $\Omega\subset \mathbb{C}$ is a connected open set and $\{K(z)\}_{z\in\Omega}$ is a holomorphic family of Fredholm operators. If $A(z_0)^{-1}$ exists for some $z_0\in\Omega$ then $z\mapsto A(z)^{-1}$ is a meromorphic family of operators for $z\in \Omega$ with poles of finite rank.
\end{theorem}

For a proof, see, e.g., \cite[Theorem C.8]{DyZw:19}.

\

For $\cH$ a Hilbert space and $B:\cH\to\cH$ a compact, self-adjoint operator, let $\{\lambda_j(B)\}_{j=1}^\infty$ denote the 
eigenvalues of $B$.
For $A$ a compact operator, let 
\beqs
s_j(A):= \sqrt{ \lambda_j(A^*A)}.
\eeqs

Let $\cH_1,\cH_2$ be Hilbert spaces.

\begin{definition}[Trace class]\label{def:trace_class}
Let $A:\cH_1\to \cH_2$ be a compact operator. 
$A$ is \emph{trace class}, $A\in \mathcal{L}_1(\cH_1;\cH_2)$, if
\beqs
\norm{A}_{\mathcal{L}_1(\cH_1;\cH_2)} := \sum_{j=1}^\infty s_j(A) <\infty.
\eeqs
\end{definition}
Observe that this definition immediately implies that if $A\in \mathcal{L}_1(\cH_1;\cH_2)$ then $(A^*A)^{1/2} \in \mathcal{L}(\cH_1;\cH_1)$ with 
\beq\label{eq:square_root}
\big\|(A^* A)^{1/2}\big\|_{\mathcal{L}_1(\cH_1;\cH_1)} = \big\|A\big\|_{\mathcal{L}_1(\cH_1;\cH_2)}.
\eeq

If $A\in \mathcal{L}(\cH_2; \cH_1)$ and $B:\cH_1\to\cH_2$ is bounded, 
then 
\beq\label{eq:composition_trace}
\norm{AB}_{\mathcal{L}_1(\cH_1;\cH_1)} \leq \norm{A}_{\mathcal{L}_1(\cH_2;\cH_1)} \norm{B}_{\cH_1 \to \cH_2};
\eeq
see, e.g., \cite[Equation B.4.7]{DyZw:19}.

If $A:\cH\to \cH$ is a finite-rank operator with non-zero eigenvalues $\{\lambda_j(A)\}_{j=0}^{n-1}$, then
\beqs
\det(I-A):= \prod_{j=0}^{n-1}\big(1- \lambda_j(A)\big).
\eeqs
This map $A\to \det(I-A)$ extends uniquely to a continuous function on $\mathcal{L}_1(\cH;\cH)$ by, e.g., \cite[Proposition B.27]{DyZw:19}.

\begin{theorem}[Properties of trace-class operators and Fredholm determinants]\label{thm:trace}
If $A\in \mathcal{L}_1(\cH;\cH)$, then the following statements are true.

(i) $I-A$ is invertible if and only if $\det(I-A) \neq 0$. 

(ii) 
\beqs
\big|\det(I-A)\big|\leq \exp( \|A\|_{\mathcal{L}_1(\cH;\cH)}).
\eeqs

(iii) 
\beqs
\big\|(I-A)^{-1}\big\|_{\cH\to \cH} \leq \frac{
\det \big( I+ (A^*A)^{1/2}\big)
}{
\big| \det(I-A) \big|
}.
\eeqs
\end{theorem}

\begin{proof}[References for the proof]
For Part (i), see, e.g., \cite[Proposition B.28]{DyZw:19}. 
For Part (ii), see, e.g., \cite[Equations B.5.11 and B.5.19]{DyZw:19}.
For Part (iii) see, e.g., \cite[Theorem 5.1, Chapter 5.1]{GoKr:69}.
\end{proof}

\section*{Acknowledgements}

The authors thank Jeffrey Galkowski (University College London) for useful discussions about the penetrable-obstacle problem
and Stephen Shipman (Louisiana State University) for useful discussions about the literature on quasi-resonances. Figures \ref{fig:QR1} and \ref{fig:QR2} were produced using code originally written by Andrea Moiola (University of Pavia) for the paper \cite{MoSp:19}.
EAS acknowledges support from EPSRC grant EP/R005591/1. JW acknowledges partial support from NSF grant DMS--2054424.

\end{appendix}

\bibliographystyle{amsalpha}
\bibliography{biblio_LSW}

\end{document}